\documentclass[10pt]{article}
\usepackage{amsfonts}
\usepackage{amsmath,amsthm,mathtools,amscd,amssymb,mathrsfs,setspace}
\usepackage[makeroom]{cancel}

\usepackage{relsize}

\usepackage{latexsym, epsf, epsfig}
\usepackage{bm}
\usepackage{color}
\usepackage[hmargin=1in,vmargin=1in]{geometry}
\usepackage{graphicx}
\usepackage{hyperref}
\usepackage{cite}
\usepackage{multirow}
\usepackage{float} 
\newcommand{\cA}{{\mathcal{A}}}

\newcommand\weakto\rightharpoonup

\theoremstyle{plain}
\newtheorem{theorem}{Theorem}[section]
\newtheorem{lemma}[theorem]{Lemma}

\newtheorem{definition}{Definition}
\theoremstyle{remark}
\newtheorem{remark}{Remark}[section]
\numberwithin{equation}{section} \numberwithin{theorem}{section}
\numberwithin{remark}{section} 

\begin{document}

\title{Semigroup Solutions for A Multilayered Filtration System}
 \author{{\small \begin{tabular}[t]{c@{\extracolsep{.8em}}c@{\extracolsep{.8em}}c}
     George Avalos &       {  Galen Richard} &{  Justin T. Webster} \\
\it Univ. Nebraska-Lincoln  \hskip.2cm  & \hskip.4cm \it Univ. of Maryland, Baltimore County   \hskip.2cm  & \hskip.4cm \it Univ. of Maryland, Baltimore County   \\
\it Lincoln, NE &  \it Baltimore, MD &\it Baltimore, MD\\
gavalos@math.unl.edu & grichar1@umbc.edu  & websterj@umbc.edu \\
\end{tabular}}}
\maketitle

\begin{abstract}
\noindent We investigate solutions to a coupled system of partial differential equations that describe a multilayered filtration system. Namely, we study the interaction of a viscous incompressible flow with bulk poroelasticity, via a poroelastic interface. The configuration consists of two 3D toroidal subdomains connected via a plate interface, which permits elastic deformation and perfusive fluid dynamics. The governing dynamics comprise Stokes equations in the bulk fluid region, Biot's equations in the bulk poroelastic region, and the recent poroplate of Mikeli\'c at the interface. Coupling occurs on the top and lower surfaces of the plate, and involves conservation of mass, stress balance, and a certain slip condition for the fluid free-flow. 

We seek strong (and mild) solutions in the Hilbert space framework via the Lumer-Phillips theorem. The resolvent analysis employs a nonstandard mixed variational formulation which captures the complex, multi-physics coupling at the interface. We explicitly characterize the infinitesimal generator associated to the linear Cauchy problem and establish the generation of a $C_0$-semigroup on a suitably chosen finite-energy  space.  With the semigroup in hand, we may treat elastic nonlinearities for plate displacements through perturbation theory. These  result parallel those for Biot-Stokes filtration systems, and complement the recently established weak solution theory for multilayer filtrations. The agency of the semigroup  straightforwardly admits structural (plate) nonlinearity into the dynamics. Future stability and regularity analyses for multilayer filtrations are also made possible by these results, as well as a comparison of spectral and regularity properties between filtration configurations, and the elucidation of the mitigating poroplate dynamics as possibly regularizing and stabilizing. 
\vskip.25cm
\noindent Acknowledgments: The first author was partially supported by NSF-DMS 1907823 and Simons Grant MP-TSM-00002793; the third author was partially supported by NSF-DMS 2307538. 
\end{abstract}

\section{Introduction} 
Poroelasticity refers to fluid flow within an elastically deformable porous media \cite{coussy,show,bgsw}. When a poroelastic model is coupled to a  fluid free flow, the resulting coupled system of partial differential equations (PDEs) is referred to as a fluid-poroelastic-structure interaction (FPSI) \cite{bcmw,yotov2,yotov1} or filtration system \cite{filt1}. These models historically arose in the context of geological motivations such as soil consolidation and resource extraction, in particular for flows adjacent to soilbeds or seabeds \cite{filt1} (and references therein). Since many biological systems can be modeled as FPSI problems, there has been a recent focus toward those applications \cite{bcmw,bw,BMW}, for instance, for blood flow through arteries and/or organ and tissue systems---see \cite{bgsw,AGW} and references therein for more discussion. The modeling of FPSI systems can, in certain circumstanced, enhanced by describing them as multilayered \cite{georgemulti} poroelastic systems. These multilayered systems are composed of a 3D bulk poroelastic layer coupled to a 3D fluid free flow, with a 2D mitigating dynamics, typically representing a reticular elastic plate \cite{jeff1,jeff2,jeff3} or poroelastic plate \cite{bcmw}. The reticular plate provides an elastic dynamics, but does not ``see" fluid dynamics and perfectly admits fluid coupling between Biot and Stokes pressures; the poroelastic plate of Mikeli\'c \cite{mikelic,ellie} is a lower-dimensional poroleastic model with elastic and fluid dynamics. The introduction of such a plate model adds additional challenges to the modeling and analysis, in particular in weak formulations and coupled elliptic systems associated to the dynamics generator. The resulting class of filtration models are multiphysics and multiscale, and careful consideration must be paid to the coupling of dynamics across the lower dimensional interface.

The work herein  studies the coupled Biot-poroplate-Stokes system. In this case the plate is the aforementioned ``2.5D" poroelastic system of recent interest \cite{ellie,mikelic}. This unique model is a standard 2D elastic plate (perhaps elastically nonlinear), which admits a 3D perfusion via an inflation variable acting in the normal  (coupling) direction and ``connecting" the top and bottom surfaces of the plate interface. We will consider the physical coupling conditions at the interface as proposed in \cite{bcmw}, where weak solutions were constructed across several physical regimes. On the lateral boundaries, we will follow suit for the recent filtration analyses \cite{AGW,AW,bcmw,jeff1,jeff2,jeff3} by considering spatially-periodic boundary conditions in the lateral directions. The goal of our work is to frame the linear filtration system as a Cauchy problem in an appropriate state space that respects the natural energy identity associated with the dynamics, as well as captures salient boundary and coupling conditions. We will postulate a dynamics generator, $\mathbf A$, and work to characterize its domain, $\mathcal D(\mathbf A)$, in such a way that the Lumer-Phillips theorem \cite{pazy} can be applied.

The chosen state space formulation will necessitate the introduction of a fluid-pressure sub-problem, permitting the elimination of the Stokes fluid pressure from the state space, as recently done in \cite{AGW}, and adapted from by now classical work such as \cite{AT}. With $\mathcal D(\mathbf A)$ defined, dissipativity will follow directly from careful integration by parts justified thereon. To prove maximality, the resolvent system will be treated through elliptic methods; in particular, a mixed variational system is formulated on a carefully chosen  auxiliary space to make use of a  form of the Babu\v{s}ka-Brezzi Lemma. As such the Lumer-Phillips yields a strongly continuous ($C_0$) semigroup of contractions corresponding to the action of the Biot-poroplate-Stokes dynamics. In the standard way \cite{pazy}---see more details in \cite{AGW}---the semigroup gives rise to strong solutions (for data in $\mathcal D(\mathbf A)$), as well as finite energy solutions as mild and weak solutions (which will coincide as in \cite{AW}). 

While this multilayered system was investigated in \cite{bcmw}, only the point of view of existence of weak solutions was taken. The central contributions also included the weak formulation of the coupled dynamics itself,  as well as a weak-strong uniqueness criterion. Certain nonlinear and degenerate regimes were accommodated. On the other hand, here, we have chosen to utilize semigroup theory yielding full well-posedness of strong and mild solutions. Regularity of strong solutions is encapsulated in the characterization of the generator's domain. Structural nonlinearities---here the von K\'arm\'an plate nonlinearity---can be easily folded into the analysis through standard semigroup perturbation results \cite{pazy}. Moreover, strong solutions can be used as approximants to construct weak solutions (even in degenerate regimes), which will be unique using the recent semigroup methodology of \cite{AW}. Lastly, the resolvent analysis performed in the maximality step yields the starting point for studying the long-time stability properties of the multilayer dynamics (such as in \cite{georgemulti}) via spectral analysis. With this multilayer filtration semigroup in hand, a forthcoming work can compare these qualitative properties (e.g., regularity and stability) with the Biot-Stokes filtration of \cite{AGW}. 

We briefly indicate how the linear semigroup framework accommodates physically relevant nonlinearities in the mitigating plate dynamics. 
Large-deflection effects are possible and relevant for FPSI modeling. In particular, for \emph{moving-domain} formulations for the surrounding bulk media, i.e., $\Omega_b=\Omega_b(t)$ and $\Omega_f=\Omega_f(t)$ as in \cite{jeff1,jeff2,jeff3}, the evolution of the domains can give rise to an interface geometry which is non-constrant, and, in particular, not a perturbation of a flat, interface domain. To address such scenarios, the mathematical approach must treat geometric aspects of the interface. In particular, a plate (reticular or poroelastic) model with intrinsic large-deflection structure (for instance in the sense of von K\'arm\'an) is desirable to extend modeling efficacy. 
Here,  we restrict our attention to fixed bulk geometries, but we utilize the semigroup to accommodate a cubic-type nonlinear elastic effect \emph{within the plate dynamics}.
 The argument we provide extends to a broad class of $H^2$ locally-Lipschitz perturbations which are standard for large deflection plate dynamics (e.g., Berger and Kirchhoff plates).  
At the same time, this should be viewed as a first step toward a fully nonlinear poro-plate theory with perfusion; in forthcoming work,  \cite{nonporoplate}, a nonlinear analogue of the Mikeli\'c poro-plate model will be developed and analyzed.

\section{Informal Statement of Main Result and Discussion}
\label{sec:main-result-informal}

The central work of this paper is to properly formulate the coupled Biot-poroplate-Stokes dynamics as an abstract Cauchy problem on a natural finite-energy  state space, $\mathbf X$, and to prove that the associated evolution operator, $\mathbf A$ is the infinitesimal generator of a  $C_0$-contraction semigroup on $\mathbf X$. Concretely, we define the action of $\mathbf A$ on its domain $\mathcal D(\mathbf A)$ so that solutions of $\dot{\mathbf y}(t)=\mathbf A\mathbf y(t)$ correspond to the linear multilayer filtration (in the absence of forcing), with the prescribed interface and boundary conditions. Generation of a $C_0$-semigroup then yields: for each $\mathbf y_0\in\mathbf X$ a unique mild (finite-energy) solution, and for each $\mathbf y_0\in\mathcal D(\mathbf A)$ a unique strong solution. Subsequently, we use perturbation theory to accommodate the von K\'arm\'an large deflection plate dynamics, yielding strong and mild solutions in this nonlinear case. 

\subsection{Novelty of the Work}
While the underlying filtration configuration and weak formulation originate in prior work, the present paper contributes a \emph{semigroup-level} well-posedness theory for the  multilayer system. This yields new analytical tools that are not visible at the purely weak-solution level. The main points are as follows:
\begin{itemize}
  \item \emph{Full well-posedness for the Biot-poroplate-Stokes system.}
  We provide an explicit generator $\mathbf A$ and a rigorous characterization of $\mathcal D(\mathbf A)$ that encodes the multi-physics interface conditions (mass conservation, stress transmission, and slip) at the top and bottom surfaces of the poro-plate. We have unique weak and strong solutions in the inertial and non-degenerate regime, and strong solutions may subsequently be used as approximants to obtain weak solutions. 

  \item \emph{Nonstandard resolvent analysis, tailored to the poroplate coupling.}
  The mediating poroelastic plate introduces additional (and coupled) surface unknowns and traces, producing an elliptic resolvent system that does not reduce to the standard Stokes-Biot transmission problem. To overcome this, we construct a mixed variational system on an auxiliary space and invoke Babu\v{s}ka-Brezzi to obtain a unique weak resolvent solution, which is then shown \emph{a posteriori} to lie in $\mathcal D(\mathbf A)$.

  \item \emph{Fluid-pressure elimination in the multilayer setting.}
  As in coupled fluid-structure semigroup formulations, we avoid carrying the Stokes pressure as a state variable by introducing a pressure subproblem that depends on the remaining unknowns (cf.\ \cite{AT,AGW}). This substantially streamlines the state space and clarifies the energy identity, and was not used (or needed) in previous weak solution analysis.

  \item \emph{A platform for perturbations and for spectral/stability and regularity analyses.}
  Once the linear semigroup is established, structural nonlinearities in the plate (e.g.\ von K\'arm\'an type) can be incorporated \cite{pazy}. Moreover, the resolvent structure developed here is the natural starting point for spectral and asymptotic stability investigations in the spirit of \cite{georgemulti}.
\end{itemize}

\subsection{Comparison with Established Literature} 
The present analysis is most naturally positioned between three quite recent strands of the literature:

\smallskip
\noindent\emph{(i) Weak solution theory for multilayer filtration:} 
In \cite{bcmw} the \emph{same coupled multilayer model} (Biot in the bulk poroelastic region, Stokes in the bulk fluid region, and a poroelastic plate interface) is treated from the standpoint of weak solutions, including regimes with quasi-static components and certain nonlinear/degenerate effects. The approach there is variational and time-discretization based, yielding weak existence (and a uniqueness criterion, though not proper uniqueness). No regularity theory is pursued. By contrast, the present work targets inertial and non-degenerate regime and proves \emph{full linear semigroup well-posedness} (mild and strong solutions) by identifying the generator and its domain. This necessarily requires a sharper encoding of interface and boundary conditions (in particular, regularity consistent with $\mathcal D(\mathbf A)$) and leads to a different set of elliptic tools in the maximality step as well as strong solutions and uniqueness of all classes of solutions. 

\smallskip
\noindent\emph{(ii) The decoupled poroelastic plate:}
The paper \cite{ellie} focuses on the evolution of the poroelastic plate from \cite{mikelic} itself, emphasizing existence/uniqueness for the plate dynamics (including quasi-static formulations) and, in an inertial setting, a semigroup-based treatment of the plate generator. Our work may be viewed as extending this perspective to the \emph{fully coupled} filtration configuration, where the plate interacts simultaneously with a Stokes flow above and a Biot medium below. This coupling creates a transmission problems with coupled traces on \emph{both} sides of the plate and necessitates a mixed formulation at the resolvent level that has no analogue in the decoupled plate analysis. 

\smallskip
\noindent\emph{(iii) Semigroup theory for Biot-Stokes without a mediating plate:}
The work \cite{AGW} establishes a semigroup framework for the coupled 3D Biot/Stokes system \emph{without} a lower-dimensional interface layer. This yields strong and mild well-posedness of that filtration system, and weak solutions are constructed in the limit. In \cite{AW}, provides full well-posedness of weak solutions for Biot-Stokes dynamics. In the case of this filtration, interface conditions differ substantially and the analysis reduces to a two-domain transmission, similar to previous work on Stokes-Lam\'e coupling. Our model here introduces an additional ``2.5D" poroelastic dynamics at the interface, producing a genuinely three-component coupling (fluid/bulk-poro/plate) and a correspondingly richer resolvent system, as motivated by \cite{georgemulti}. Nevertheless, we adapt two key structural ideas from \cite{AGW}: (a) eliminating the Stokes pressure from the state space via an auxiliary pressure problem (cf.\ also the classical semigroup FSI formulations in \cite{AT}); and (b) establishing maximality through a mixed variational formulation and a Babu\v{s}ka--Brezzi argument. The novelty here is that these mechanisms must be wholly redesigned to accommodate the poro-plate variables and their two-sided coupling.


\section{PDE Model and Modeling Background}
\subsection{Spatial Domain and Physical Quantities}
The spatial configuration of interest is pictured below. It consists of a 3D toroidal domain, with a 2D interface separating the bulk regions.

\begin{figure}[H]
\includegraphics[width=8cm]{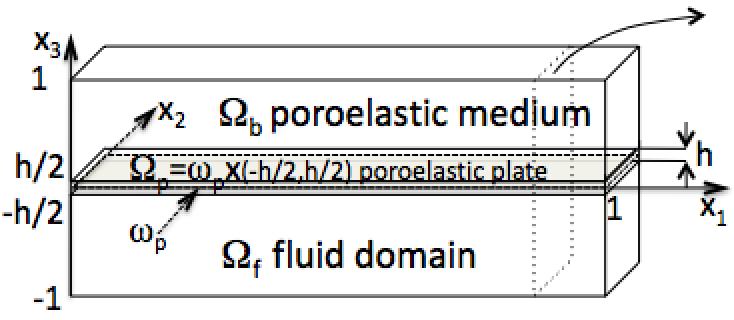} \quad
\includegraphics[width=6cm]{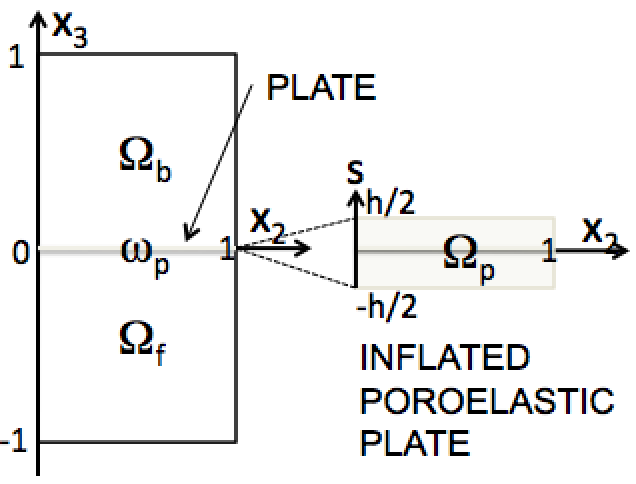}
\caption{Left: 3D domain $\Omega$, including plate pressure $\Omega_p$. Right: 2D vertical cross-section through domain $\Omega$. (Original image from \cite{bcmw}.)}\label{fig:domains}
\end{figure}

\noindent The set $\Omega_b := (0,1)^2 \times (0,1)$ is the Biot domain, $\Omega_f := (0,1)^2 \times (0,-1)$ is the free fluid domain. There are two domains used for the poroelastic plate; $\omega_p := (0,1)^2 \times \{0\}$ is the displacement domain, and $\Omega_p := (0,1)^2 \times (-h/2, h/2)$ is the inflation domain used for the plate pressure. Note that for the pressure domain, the inflation variable $s$ is used, which, although parallel with, is distinct from $\mathbf{x}_3$. This is a specific convention which we follow, due to Mikeli\'c \cite{bcmw}. The interfaces of interest are the elastic plate interface $\Gamma_I := \omega_p$, and the top and bottom plate pressure interfaces $\omega_p^+ = \partial \Omega_p \vert_{s =  h/2}$ and $\omega_p^- = \partial \Omega_p \vert_{s = -h/2}$. The top and bottom boundaries are $\Gamma_b := \partial \Omega_b \vert_{\mathbf{x}_3 = 1}$ and $\Gamma_f := \partial \Omega_f \vert_{\mathbf{x}_3 = -1}$. Letting $\Omega := \Omega_b \cup \omega_p \cup \Omega_f$, its lateral faces are then $\Gamma_{lat} := \partial \Omega \setminus (\Gamma_b \cup \Gamma_f)$. Finally the plate boundary is $\Gamma_c := \partial \omega_p$.\\ \\
The physical constants of interest here are---see \cite{coussy,show,filt1,mikelic} for more information:
\begin{itemize}
    \item $(\lambda_b, \mu_b, \rho_b, \alpha_b, c_b, k_b)$, where $\lambda_b, \mu_b$ are the Lam\'e constants of the biot bulk Biot region, $\rho_b$ is the density of the poroelastic matrix, $\alpha_b$ is the Biot-Willis coefficient, $c_b$ is the constrained storage coefficient, and $k_p$ is the constant permeability of the poroelastic matrix.
    \item $(D, \gamma, \rho_p, \alpha_p, c_p, k_p)$, where $D > 0$ is the elastic stiffness of the plate, and $\gamma > 0$ is an coefficient which is added for coercivity. The other constants act the same as in the 3D Biot region, but for the poroelastic plate.
    \item $(\rho_f, \mu_f)$, are the fluid density and  dynamic viscosity.
\end{itemize}
The variables utilized in the system are
\begin{itemize}
    \item $(\boldsymbol{\eta}, p_b)$, poroelastic displacement and pressure defined on $\Omega_b$
    \item $(w, p_p)$, transverse displacement of the plate, and the poroelastic plate pressure, defined on $\omega_p$ and $\Omega_p$ respectively;
    \item $(\mathbf{u}, p_f)$, Stokes fluid velocity and pressure defined on $\Omega_f$.
\end{itemize}

We utilize an adjoint pair $\mathcal{K}, \tilde{\mathcal{K}}$, for the poroelastic plate, which are defined as follows (see \cite{ellie}):
\begin{definition}
The moment operator $\mathcal{K} : L^2(\Omega_p) \rightarrow L^2(\omega_p)$, which has continuous action
$$\mathcal{K}(p) = \int_{-\frac{h}{2}}^{\frac{h}{2}} x_3p \ dx_3.$$
Moreover, we define $\tilde{\mathcal{K}} : L^2(\omega_p) \rightarrow L^2(\Omega_p)$ with action
$$\tilde{\mathcal{K}}(q) = x_3 q.$$
These operators are adjoint, in that:
$$(\mathcal{K}(p), q)_{L^2(\omega_p)} = (p, \mathcal{\tilde{K}}(q))_{L^2(\Omega_p)}; \ \forall p \in L^2(\Omega_p), ~q \in L^2(\omega_p).$$
\end{definition}
To define the PDEs of interest, we will utilize the following stress tensors. First we define the symmetrized gradient $\mathbf{D}$ by its action~~
$\mathbf{D}(\mathbf{v}) = \frac{1}{2}(\nabla \mathbf{v} + (\nabla \mathbf{v})^T).$ This allows us to introduce the elastic stress. 
\begin{definition}
    We invoke the Saint Venant-Kirchhoff elastic stress tensor \cite{kesavan},
    $$\sigma^E(\boldsymbol{\eta}) = 2 \mu_b \mathbf{D}(\boldsymbol{\eta}) + \lambda_b (\nabla \cdot \boldsymbol{\eta}) \mathbf{I}.$$
    This gives rise to the full poroelastic stress tensor \cite{show,coussy}
    $$\sigma_b(\boldsymbol{\eta}, p_b) = \sigma^E(\boldsymbol{\eta}) - \alpha_b p_b \mathbf{I}.$$
    Finally we utilize the standard fluid stress tensor
    $$\sigma_f(\mathbf{u}, p_f) = 2 \mu_f \mathbf{D}(\mathbf{u}) - p_f \mathbf{I}.$$
\end{definition}
\begin{remark}\label{sym}
Note that when $\nabla \cdot \mathbf{v} = 0$,
$$2 \, \text{div}\{\mathbf{D}(\mathbf{v})\} = \Delta \mathbf{v} + \nabla (\nabla \cdot \mathbf{v}) = \Delta \mathbf{v}.$$
Also from the symmetries of $\mathbf{I}$ and $\mathbf{D}(\cdot)$ we know that
\begin{align*}
    (\sigma^E(\boldsymbol{\eta}), \mathbf{D}(\mathbf{w}))_{L^2} &= 2 \mu_b (\mathbf{D}(\boldsymbol{\eta}), \mathbf{D}(\mathbf{w}))_{L^2} + \lambda_b ((\nabla \cdot \boldsymbol{\eta}) \mathbf{I}, \mathbf{D}(\mathbf{w}))_{L^2} \\
        &= 2 \mu_b (\mathbf{D}(\boldsymbol{\eta}), \nabla \mathbf{w})_{L^2} + \lambda_b ((\nabla \cdot \boldsymbol{\eta})\mathbf{I}, \nabla \mathbf{w})_{L^2} \\
        &= (\sigma^E(\boldsymbol{\eta}), \nabla \mathbf{w})_{L^2}.
\end{align*}\end{remark}
Using the  symmetry in Remark \ref{sym}, and the definition of the Frobenius tensor product, we have 
\begin{equation} (\sigma^E(\boldsymbol{\eta}), \mathbf{D}(\mathbf{w}))_{L^2} = (\sigma^E(\mathbf{w}), \mathbf{D}(\boldsymbol{\eta}))_{L^2}.\end{equation}
We also use the convention that the elementary unit vectors $\mathbf{e}_i$ run in the same direction as $\mathbf{x}_i$ (see Figure \ref{fig:domains}). The symbol $\boldsymbol{\tau}=\mathbf e_i$ will denote (tangential) cases where $i = 1,2$.

\subsection{Dynamics Model}
We restate the model from \cite{bcmw}, adapted here. Namely, the multi-layered system under consideration is:
\begin{align}
\rho_b \boldsymbol{\eta}_{tt} - \nabla \cdot \sigma^{E}(\boldsymbol{\eta}) + \alpha_b \nabla p_b &= \mathbf{F}_b &\ \text{on} \ (0,T) \times \Omega_b \label{*.a} \\
[c_b p_b + \alpha_b \nabla \cdot \boldsymbol{\eta}]_t - \nabla \cdot (k_b \nabla p_b) &= S &\ \text{on} \ (0,T) \times \Omega_b \label{*.b} \\
\rho_p w_{tt} + D \Delta_{\omega_p}^2 w + \gamma w + \alpha_p \Delta_{\omega_p} \mathcal{K}(p_p) &= F_p &\ \text{on} \ \omega_p \label{*.c} \\ 
[c_p p_p - \alpha_p \tilde{\mathcal{K}}(\Delta_{\omega_p} w)]_t - \partial_s (k_p \partial_s p_p) &= 0 &\ \text{on} \ \Omega_p \label{*.d} \\
\rho_f \mathbf{u}_t - \mu_f \Delta \mathbf{u} + \nabla p_f &= \mathbf{F}_f &\ \text{on} \ \Omega_f \label{*.e} \\
\nabla \cdot \mathbf{u} &= 0 &\ \text{on} \label{*.f} \ \Omega_f.
\end{align}
The plate forcing term above will be denoted by ${F}_p=\sigma_b \mathbf{e}_3 \cdot \mathbf{e}_3 \vert_{\omega_p} - \sigma_f \mathbf{e}_3 \cdot \mathbf{e}_3 \vert_{\omega_p}$ at times below. Above, we write $k_b$ and $k_p$ within the operators, indicating that future work may may take them to be variable or nonlinear coefficients \cite{ellie,bcmw,bgsw}. Although we do not specify it here, there will be Cauchy data associated with this system as well. The coupling conditions for the multilayered Biot-Stokes solution variables on $(0,T) \times \Gamma_I$ are:
\begin{align}
\langle 0,0,w \rangle &= \boldsymbol{\eta} \vert_{\omega_p} & \text{(continuity of vertical displacements)} \label{*.g} \\
\beta \mathbf{u} \vert_{\omega_p} \cdot \boldsymbol{\tau} &= - [\sigma_f \vert_{\omega_p} \mathbf{e}_3] \cdot \boldsymbol{\tau} & \text{(Beavers-Joseph-Saffman; tangential velocity slip])} \label{*.h} \\
k_p \partial_s p_p \vert_{\omega_p^+} &= k_b \partial_{\mathbf{x_3}} p_b \vert_{\omega_p} & \text{(continuity of vertical filtration velocities)} \label{*.i} \\
p_p \vert_{\omega_p^+} &= p_b \vert_{\omega_p} & \text{(continuity of pore pressures)} \label{*.j} \\
w_{t} - (\mathbf{u} \cdot \mathbf{e}_3)\vert_{\omega_p} &= k_p \partial_s p_p \vert_{\omega_p^-} & \text{(continuity of filtration velocities)} \label{*.k} \\
-\sigma_f \mathbf{e}_3 \cdot \mathbf{e}_3 \vert_{\omega_p} &= p_p \vert_{\omega_p^-} & \text{(continuity of vertical fluid stress)}. \label{*.l}
\end{align}
 For the non-interactive boundary conditions, we begin by asserting that the plate will be taken to be clamped at its boundary, i.e.,
$$w=\partial_{\nu}w = 0, ~\text{ on }~\partial \omega_p.$$
\begin{remark} While the fluid and bulk poroelastic domains are toroidal (i.e., taken with periodic boundary conditions), clamping the plate eliminates certain non-critical technical issues for strong solutions. This type of assumption has been made in other multilayer analyses \cite{bcmw,jeff1}.\end{remark}

For the lateral sides, we introduce a few relevant spaces: With the domain $\Omega$ here denoting either $\Omega_b$ or $\Omega_f$, we define the following Sobolev spaces which have lateral periodicity; namely,
$$H_\#^k(\Omega) = \{ \phi \in H^k(\Omega) ~:~ \gamma_j[\phi] \vert_{x_i = 0} = \gamma_j [\phi] \vert_{x_i = 1}; ~ j = 0,1,k-1, ~ i = 1,2 \},$$
where $i$ refers to the indices as described in \figurename{~\ref{fig:domains}}. 
Here, with $\Gamma _{l}$ denoting a lateral side of $\Omega $, $\left. \gamma
_{m}\right\vert _{\Gamma _{j}}$ is the usual Sobolev Trace mapping, restricted to a boundary subdomain;
i.e., for $\varphi \in H^{m}(\Omega )$ sufficiently smooth, $m\geq 1$,%
\[
\left. \gamma _{m-1}(\varphi )\right\vert _{\Gamma _{l}}=\left. \frac{%
\partial ^{m-1} \varphi}{\partial \mathbf{n}^{m-1}}\right\vert _{\Gamma _{l}};
\]%
see e.g., \cite[Corollary 4.3, p.21]{BDM}.
Subsequently, We define the spaces
$$H^k_{\#,*}(\Omega) = \{ f \in H^k_\#(\Omega) ~:~ f \vert_{\Gamma} = 0 \},$$
where, again, $\Omega$ is either $\Omega_f$ (so $\Gamma$ is $\Gamma_f$), or $\Omega$ is $\Omega_b$ (so  $\Gamma$ is $\Gamma_b$).

\smallskip
In addition, we will denote

\begin{align}\nonumber
H_{\#}^{-\frac{1}{2}}(\Omega _{f})= & ~\big\{ T\in \mathcal{D}'(\partial \Omega _{f}):
\left\langle \left. T\right\vert_{\Gamma
_{\alpha}},g\right\rangle _{\Gamma _{\alpha}}=
\left\langle \left. T\right\vert_{\Gamma
_{\beta}},g\right\rangle _{\Gamma _{\beta}},\text{ }\forall g\in
H_{\#}^{1}(\Omega _{f}) \\ 
 &\hskip3cm \text{and for each pair of opposing lateral sides }\Gamma _{\alpha}\text{ and }%
\Gamma _{\beta} \big\}
\end{align}
The analogous definition holds for $H_{\#}^{-\frac{1}{2}}(\Omega _{b})$.

Given these spaces, we will then specify the regularity for the solution variables, together with the boundary conditions away from the boundary interface $\Gamma_I$:
\begin{align*}
    \boldsymbol{\eta} \in \mathbf{H}^1_{\#,*}(\Omega_b), ~~p_b \in H^1_{\#,*}(\Omega_b);  \hskip1cm
    \mathbf{u} \in \mathbf{H}^1_{\#,*}(\Omega_f); \hskip1cm
    w \in H^2_{0}(\omega_p), ~~p_p \in H^{1}(\Omega _{p}). 
    \end{align*}
This regularity will be justified in by the energy identity in the sequel; we will also later specify the regularity of the forcing terms.

\subsection{Energy Identity}

\begin{definition}
    With respect to the Hilbert space $\mathbf{H}^1$, we define the norm and symmetric bilinear form,
    $$\Vert \cdot \Vert_E^2 = a_E(\cdot, \cdot) = (\sigma^E(\cdot), \mathbf{D}(\cdot))_{L^2(\Omega_b)},$$
    which is indeed equivalent to the standard $\mathbf{H}^1$-norm by 
    Korn and Poincar\'e's inequalities \cite{kesavan}.
\end{definition}
To obtain the formal energy identity, we multiply the equations (\ref{*.a})--(\ref{*.f}) by $(\boldsymbol{\eta}_{t}, p_b, w_{t}, p_p, \mathbf{u})$ respectively, and integrate in the appropriate space. Finally we set $\mathbf{F}_b = \mathbf{F}_f = \mathbf{0}$ and $S = 0$ here. Assuming necessary regularity, our formal calculations result in:
\begin{align*}
(E.1)~~~~
    0 &= \mathlarger{(} \rho_b \boldsymbol{\eta}_{tt} - \nabla \cdot \sigma^E(\boldsymbol{\eta}) + \alpha_b \nabla p_b, \boldsymbol{\eta}_{t} \mathlarger{)}_{L^2(\Omega_b)} \\
    &= \frac{1}{2} \frac{d}{dt} \mathlarger{[} \rho_b \Vert \boldsymbol{\eta}_{t} \Vert^2_{\Omega_b} + \Vert \boldsymbol{\eta} \Vert^2_E \mathlarger{]} + \int_{\Gamma_I} \sigma^E(\boldsymbol{\eta}) \mathbf{e}_3 \cdot \boldsymbol{\eta}_{t} d\Gamma_I+ (\alpha_b \nabla p_b, \boldsymbol{\eta}_{t})_{L^2(\Omega_b)} \\
    &= \frac{1}{2} \frac{d}{dt} \mathlarger{[} \rho_b \Vert \boldsymbol{\eta}_{t} \Vert^2_{\Omega_b} + \Vert \boldsymbol{\eta} \Vert^2_E \mathlarger{]} + \int_{\Gamma_I} \sigma^E(\boldsymbol{\eta}) \mathbf{e}_3 \cdot \boldsymbol{\eta}_{t} d\Gamma_I - \cancel{(\alpha_b p_b, \nabla \cdot \boldsymbol{\eta}_{t})_{L^2(\Omega_b)}} - \alpha_b \int_{\Gamma_I} p_b \mathbf{e}_3 \cdot \boldsymbol{\eta}_{t} d\Gamma_I.\\ \\
   (E.2)~~~~ 0 &= \mathlarger{(}\mathlarger{[}c_b p_b + \alpha_b \nabla \cdot \boldsymbol{\eta}\mathlarger{]}_t - \nabla \cdot (k_b \nabla p_b), p_b\mathlarger{)}_{L^2(\Omega_b)} \\
    &= \frac{1}{2} \frac{d}{d t} \mathlarger{[} c_b \Vert p_b \Vert^2_{\Omega_b} \mathlarger{]} + \cancel{(\alpha_b \nabla \cdot \boldsymbol{\eta}_{t}, p_b)_{L^2(\Omega_b)}} + k_b \Vert \nabla p_b \Vert^2_{\Omega_b} + \int_{\Gamma_I} k_b \nabla p_b \cdot \mathbf{e}_3 p_b d\Gamma_I.
    \end{align*}
 In addition, with 
\begin{equation}
F_p = \sigma _{b}\mathbf{e}_{3}\cdot \mathbf{e}_{3}-\sigma _{f}\mathbf{e%
}_{3}\cdot \mathbf{e}_{3},  \label{flux}
\end{equation} 
we have
   \begin{align*}
   (E.3)~~~~ (F_p, w_{t})_{L^2(\omega_p)} &= \mathlarger{(}\rho_p w_{tt} + D \Delta^2_{wp} w + \gamma w + \alpha_p \Delta_{\omega_p}\mathcal{K}(p_p), w_{t}\mathlarger{)}_{L^2(\omega_p)} \\
    &= \frac{1}{2} \frac{d}{dt} \mathlarger{[}\rho_p \Vert w_{t} \Vert^2_{\omega_p} + \gamma \Vert w \Vert^2_{\omega_p} \mathlarger{]} + \mathlarger{(}\Delta_{\omega_p}\mathlarger{[}D \Delta_{\omega_p} w + \alpha_p \mathcal{K}(p_p)\mathlarger{]}, w_{t}\mathlarger{)}_{L^2(\omega_p)} \\
    &= \frac{1}{2} \frac{d}{dt} \mathlarger{[}\rho_p \Vert w_{t} \Vert^2_{\Omega_b} + \gamma \Vert w \Vert^2_{\Omega_b} \mathlarger{]} + \mathlarger{(}D \Delta_{\omega_p} w, \Delta_{\omega_p} w_{t}\mathlarger{)}_{L^2(\omega_p)} + \bcancel{\mathlarger{(}\alpha_p \mathcal{K}(p_p), \Delta_{\omega_p} w_{t} \mathlarger{)}_{L^2(\omega_p)}}. \\ \\
   (E.4)~~~~ 0 &= \mathlarger{(}\mathlarger{[}c_p p_p - \alpha_p \tilde{\mathcal{K}}(\Delta_{\omega_p} w)\mathlarger{]}_t - \partial_s k_p \partial_s p_p, p_p\mathlarger{)}_{L^2(\Omega_p)} \\
    &= \frac{1}{2}\frac{d}{dt}\mathlarger{[} c_p \Vert p_p \Vert^2_{\Omega_p} \mathlarger{]} + k_p \Vert \partial_s p_p \Vert^2_{\Omega_p} - \bcancel{\mathlarger{(}\alpha_p \tilde{\mathcal{K}}(\Delta_{\omega_p} w_t), p_p\mathlarger{)}_{L^2(\Omega_p)}} - \int_{\omega_p^+} k_p (\partial_s p_p) p_p d\omega_p^+ \\
         & ~~~~~~~~ + \int_{\omega_p^-} k_p (\partial_s p_p) p_p  d\omega_p^- . \\ \\
   (E.5)~~~~ 0 &= \mathlarger{(}\rho_f \mathbf{u}_t - \mu_f \Delta \mathbf{u} + \nabla p_f, \mathbf{u}\mathlarger{)}_{\mathbf{L}^2(\Omega_f)} \\
    &= \frac{1}{2}\frac{d}{dt} \mathlarger{[}\rho_f \Vert \mathbf{u} \Vert^2_{\Omega_f} \mathlarger{]} + \mu_f \Vert \nabla \mathbf{u} \Vert^2_{\Omega_f}  - \int_{\Gamma_I} \mu_f \mathbf{e}_3 \cdot \nabla \mathbf{u}  \cdot \mathbf{u} d\Gamma_I + \int_{\Gamma_I} p_f \mathbf{u} \cdot \mathbf{e}_3 d\Gamma_I.
\end{align*}
At this point, we focus on the boundary terms. First off, we utilize that $\mathcal{K}$ and $\tilde{\mathcal{K}}$ are adjoint; as such,
\begin{equation}
\mathlarger{(}\alpha_p \mathcal{K}(p_p), \Delta_{\omega_p} w_{t}\mathlarger{)}_{L^2(\omega_p)} - \mathlarger{(}\alpha_p \tilde{\mathcal{K}}(\Delta_{\omega_p} w_t), p_p\mathlarger{)}_{L^2(\Omega_p)} = 0.  \label{b1}
\end{equation} 
Next we utilize the definition of $\sigma_f$; namely,
$$- \int_{\Gamma_I} \mu_f (\mathbf{e}_3 \cdot \nabla \mathbf{u}) \cdot \mathbf{u} + \int_{\Gamma_I} p_f \mathbf{u} \cdot \mathbf{e}_3 = - \int_{\Gamma_I} \sigma_f \cdot \mathbf{e}_3 \mathbf{u}$$
and the definition of $F_p$ in (\ref{flux}), to see that
\begin{equation}
-\int_{\Gamma _{I}}\sigma _{f}\mathbf{e}_{3}\cdot \mathbf{u}d\Gamma
_{I}-\int_{\Gamma _{I}}(F_p)w_{t}d\Gamma _{I} \\ 
= \int_{\Gamma _{I}}\sigma _{f}\mathbf{e}_{3}\cdot \mathbf{e}%
_{3}[w_{t}-(\mathbf{e}_{3}\cdot \mathbf{u})]d\Gamma _{I}-\int_{\Gamma
_{I}}\sigma _{f}\mathbf{e}_{3}\cdot \mathbf{\tau }(\mathbf{\tau }\cdot 
\mathbf{u})]d\Gamma _{I}  -\int_{\Gamma _{I}}\sigma _{b}\mathbf{e}_{3}\cdot \mathbf{e}%
_{3}w_{t}d\Gamma _{I}.%
\label{b2}
\end{equation}

To this relation, we apply these boundary interface conditions:
\begin{equation} 
w_{t} - (\mathbf{u} \cdot \mathbf{e}_3)\vert_{\omega_p} = k_p \partial_s p_p \vert_{\omega_p^-}; ~~
-\sigma_f \mathbf{e}_3 \cdot \mathbf{e}_3 \vert_{\omega_p} = p_p \vert_{\omega_p^-}; ~~
\beta \mathbf{u} \vert_{\omega_p} \cdot \boldsymbol{\tau} = - \sigma_f \vert_{\omega_p} \mathbf{e}_3 \cdot \boldsymbol{\tau}.
\end{equation}
These applied to (\ref{b2}) yield
\begin{equation}
-\int_{\Gamma _{I}}\sigma _{f}\mathbf{e}_{3}\cdot \mathbf{u}d\Gamma
_{I}-\int_{\Gamma _{I}}(F_p)w_{t}d\Gamma _{I}= -\int_{\omega _{p}^{-}}k_{p}(\partial _{s}p_{p})p_{p}d\omega
_{p}^{-}+\beta \int_{\Gamma _{I}}\left\vert \mathbf{u}\cdot \mathbf{\tau }%
\right\vert ^{2}d\Gamma _{I} -\int_{\Gamma _{I}}\sigma _{b}\mathbf{e}_{3}\cdot \mathbf{e}%
_{3}w_{t}d\Gamma _{I}.%
\label{b4}
\end{equation}
So, at this point, one can combine the relations (E.1)--(E.5) with (\ref{b4}). Using
\begin{eqnarray}
\mathcal{E}(t) &\equiv &\rho _{b}\left\Vert \mathbf{\eta }_{t}(t)\right\Vert
_{\Omega _{b}}^{2}+\left\Vert \mathbf{\eta }(t)\right\Vert
_{E}^{2}+c_{b}\left\Vert p_{b}(t)\right\Vert _{\Omega _{b}}^{2}  \nonumber \\
&&+\rho _{p}\left\Vert w_{t}(t)\right\Vert _{\omega _{p}}^{2}+\left\Vert D^{%
\frac{1}{2}}\Delta _{\omega _{p}}w(t)\right\Vert _{\omega
_{p}}^{2}+c_{p}\left\Vert p_{p}(t)\right\Vert _{\Omega _{p}}^{2}+\rho
_{f}\left\Vert \mathbf{u}(t)\right\Vert _{\Omega _{f}}^{2}.  \label{b6}
\end{eqnarray}
we have, 
for $t>0$,

\begin{eqnarray}
0=&&\frac{1}{2}\frac{d}{dt}\mathcal{E}(t)+\int_{\Gamma _{I}}\sigma _{b}\mathbf{%
e}_{3}\cdot \mathbf{\eta }_{t}d\Gamma _{I}+k_{b}\left\Vert \nabla
p_{b}\right\Vert _{\Omega _{b}}^{2}+k_{p}\left\Vert \partial
_{s}p_{p}\right\Vert _{\Omega _{p}}^{2}+\int_{\Gamma _{I}}k_{b}\nabla p_{b}\cdot \mathbf{e}_{3}p_{b}d\Gamma
_{I}  \nonumber \\
&&-\int_{\omega _{p}^{+}}k_{p}(\partial _{s}p_{p})p_{p}d\omega
_{p}^{+}+\mu _{f}\left\Vert \nabla \mathbf{u}\right\Vert _{\Omega _{f}}^{2} +\beta \int_{\Gamma _{I}}\left\vert \mathbf{u}\cdot \mathbf{\tau }%
\right\vert ^{2}d\Gamma _{I}-\int_{\Gamma _{I}}\sigma _{b}\mathbf{e}%
_{3}\cdot \mathbf{e}_{3}w_{t}d\Gamma _{I}  \nonumber 
 \label{b5}
\end{eqnarray}

To finish the derivation of the energy relation: using the interface BC,
$$\langle 0,0,w \rangle = \boldsymbol{\eta} \vert_{\omega_p},$$
we have that

\begin{equation}
- \int_{\Gamma_I} \sigma_b \mathbf{e}_3 \cdot \boldsymbol{\eta}_{t} d\Gamma_{I} + \int_{\Gamma_I} \sigma_b \mathbf{e}_3 \cdot \mathbf{e}_3 w_{t} d\Gamma_I = 0. \label{b7}
\end{equation}
Moreover, using the interface conditions,
\begin{align*}
k_p \partial_s p_p \vert_{\omega_p^+} &= k_b \partial_{x_3} p_b \vert_{\omega_p}; ~~
p_p \vert_{\omega_p^+} = p_b \vert_{\omega_p},
\end{align*}
we have that
\begin{equation}
- \int_{\omega_p^+} k_p (\partial_s p_p) p_p + \int_{\Gamma_I} k_b \nabla p_b \cdot \mathbf{e_3} p_b = 0. \label{b8}
\end{equation}
Applying (\ref{b7}) and (\ref{b8}) to RHS of (\ref{b5}), we have at last,
\begin{equation}
\frac{1}{2}\frac{d}{dt}\mathcal{E}(t)=-\left[ \beta \int_{\Gamma
_{I}}\left\vert \mathbf{u}\cdot \mathbf{\tau }\right\vert ^{2}d\Gamma
_{I}+\left\Vert k_{b}^{\frac{1}{2}}\nabla p_{b}\right\Vert _{\Omega
_{b}}^{2}+\left\Vert k_{p}^{\frac{1}{2}}\partial _{s}p_{p}\right\Vert
_{\Omega _{p}}^{2}+\mu _{f}\left\Vert \nabla \mathbf{u}\right\Vert _{\Omega
_{f}}^{2}\right] ,  \label{b9}
\end{equation}
where, again, the energy function $\mathcal{E}(t)$ is as given in (\ref{b6}).

This relation demonstrates the conserved quantities under the dynamics which will be captured in the finite-energy (state space) conditions on dynamic solution variables. Additionally, it indicates the dissipation present in the problem, coming from diffusion in each of the three spatial domains, as well as drag due to the Beavers-Joseph-Saffman slip condition.

\section{The Dynamics Operator}
\subsection{Fluid-Pressure Elimination}
We start by following the approach in \cite{AGW} to eliminate the associated fluid-structure pressure. Namely, in part by applying the divergence operator to both sides of the Stokes problem (\ref{*.e}), one can characterize the fluid pressure variable $p_f$---again, periodic on the lateral faces of $\Omega_f$---as the solution of the following boundary value problem:
\begin{equation}
\begin{cases}
\Delta p_{f}=0\text{ \ in }\Omega _{f} \\ 
\partial _{\mathbf{e}_{3}}p_{f}=\mu _{f}\Delta \mathbf{u}\cdot \mathbf{e}_{3}%
\text{ \ on }\Gamma _{f} \\ 
p_{f}=p_{p}(\cdot ,-h)+2\mu _{f}\mathbf{D}\left( \mathbf{u}\right) \mathbf{e}%
_{3}\cdot \mathbf{e}_{3}\text{ \ on }\Gamma _{I}%
\end{cases}
\label{bvp}
\end{equation}
\noindent
and $p_f$ is periodic on the lateral sides of $\Omega$. Above, the second relation comes about by dotting both sides of (\ref{*.e}) by the normal direction, $-\mathbf{e_3}$, and thereafter taking the Dirichlet trace on $\Gamma_f$. The third equation comes from the interface condition
$$-\sigma_f \mathbf{e}_3 \cdot \mathbf{e}_3 \vert_{\omega_p} = p_p \vert_{\omega_p^-}.$$
We formally write the solution of (\ref{bvp}) via the following ``Dirichlet'' and ``Neumann'' maps, which provide harmonic extensions of boundary functions \cite{AT}: 
$$\phi = N_f g \iff \begin{cases}
    \Delta \phi = 0 & \text{in} \ \Omega_f \\
    \partial_{\mathbf{e}_{3}} \phi = - g & \text{on} \ \Gamma_f \\
    \phi = 0 & \text{on} \ \Gamma_I \\
    \phi \text{ is periodic} & \text{on } \partial \Omega _{f}\diagdown (\Gamma _{I}\cup \Gamma _{f}) 
    \end{cases}
\quad \quad
\psi = D_I h \iff \begin{cases}
    \Delta \psi = 0 & \text{in} \ \Omega_f \\
    \partial_{\mathbf{e}_3} \psi = 0 & \text{on} \ \Gamma_f \\
    \psi = h & \text{on} \ \Gamma_I \\
    \psi \text{ is periodic} & \text{on } \partial \Omega _{f}\diagdown (\Gamma _{I}\cup \Gamma _{f}).
\end{cases}$$ 
By elliptic theory, $D_{f}\in \mathcal{L}(H^{-\frac{1}{2}}(\Gamma _{f}),L^{2}(\Omega _{f}))$ and 
$N_{f}\in \mathcal{L}(H^{-\frac{1}{2}}(\Gamma _{f}),L^{2}(\Omega _{f}))$, valid here since the bulk domains are toroidal, i.e., smooth and devoid of corners or meeting mixed boundary conditions;
see \cite{L-M}.
Therewith, we define the following operators $\Pi_i$,
$$\Pi_1 p_p = D_I[p_p \vert_{\omega_p^-}], \ \Pi_2 \mathbf{u} = D_I[(\mathbf{e}_3 \cdot [ 2 \mu_f \mathbf{D}(\mathbf{u}) \mathbf{e}_3])_{\Gamma_I}] + N_f[(\mu_f \Delta \mathbf{u} \cdot \mathbf{e}_3)_{\Gamma_f}],$$
with respective Green's maps are $G_i = - \nabla \Pi_i$. Through the agency of these maps, the incompressible pressure variable, as
the solution of (\ref{bvp}), has the representation 
\begin{equation}
p_{f}= \Pi_{1}(p_{p})+\Pi_{2}(\mathbf{u}).  \label{pf}
\end{equation}
Moreover, the Stokes PDE evolution (\ref{*.e}) (taken with zero forcing
term) can be written as%
\begin{equation}
\rho _{f}\mathbf{u}_{t}=\mu _{f}\Delta \mathbf{u}+G_{1}(p_{p})+G_{2}(\mathbf{%
u}).  \label{sub}
\end{equation}

\subsection{Constituent Spaces}
In the following, the subscripts $b$, $p$, and $f$ will be attached to spaces and variables which pertain to the Biot, poroplate, and Stokes systems, respectively.
With respect to the plate pressure equation, it will be useful, for the purposes of semigroup generation, to quantify regularity in the vertical direction, versus the horizontal directions. To this end, we will utilize the following anisotropic Sobolev spaces and corresponding norms:
$$H^{0,0,k}(\Omega_p) = \left\{ \phi \in L^2(\Omega_p) \ : \ \partial_s \phi, \dots, \partial_s^k \phi \in L^2(\Omega_p) \right\}$$
$$\Vert \phi \Vert_{H^{0,0,k}(\Omega_p)} = \left( \sum_{i=0}^k \vert \partial_s^i \phi \vert_{L^2(\Omega_p)} \right)^{1/2}.$$

To formally proceed, we will need to discuss $s$-traces for functions in the above anisotropic Sobolev spaces.

\begin{lemma}
If  $q\in H^{0,0,1}(\Omega _{p})$, then $\left.
q\right\vert _{\omega _{p}^{+}}$ is well-defined in $H^{-%
\frac{1}{2}}(\omega _{p}^{+})$. (Mutatis mutandis, for $\omega_p^{-}$.)
\end{lemma}
\begin{proof}
 Indeed, \ let $\mathbf{n}_{p}=[n_{1},n_{2},n_{s}]$ denote
the exterior unit normal with respect to geometry $\Omega _{p}$; let $\gamma
_{0}^{+}\in \mathcal{L}(H^{\frac{1}{2}}(\partial \Omega _{p}),H^{1}(\Omega
_{p}))$ be the right continuous inverse of the zero order Sobolev trace
mapping, $\gamma _{0}\in \mathcal{L}(H^{1}(\Omega _{p}),H^{\frac{1}{2}%
}(\partial \Omega _{p}))$---see \cite{kesavan} or \cite[Theorem 3.38]{Mc}. Therewith,
given $g\in H_{0}^{\frac{1}{2}+\epsilon }(\omega _{p}^{+})$, we denote its
extension by zero as $g_{ext}$ onto the rest of $\partial \Omega _{p}$. As
such, we then have, via integration by parts in $s$, the identity:
\[
\int_{\omega _{p}^{+}}qgd\omega _{p}^{+}=\int_{\partial \Omega
_{p}}qg_{ext}n_{s}d \partial \Omega _{p}=\int_{\Omega _{p}}(\partial
_{s}q)\gamma _{0}^{+}(g_{ext})d \Omega _{p}+\int_{\Omega
_{p}}q(\partial _{s}\gamma _{0}^{+}(g_{ext}))d \Omega _{p},
\]%
whence we obtain through the surjectivity of the trace operator, after an extension by continuity, the well-defined map,%
\begin{equation}
q\in H^{0,0,1}(\Omega _{p})\text{ }\Longrightarrow \left\{ \left.
q\right\vert _{\omega _{p}^{+}}\in H^{-\frac{1}{2}}(\omega _{p}^{+})\text{
and }\left. q\right\vert _{\omega _{p}^{-}}\in H^{-\frac{1}{2}}(\omega
_{p}^{-})\right\} 
 \label{trace}
\end{equation}
\end{proof}
In the same manner, one can obtain the higher order trace mapping,%
\begin{equation}
q\in H^{0,0,2}(\Omega _{p}) \Longrightarrow \left\{ \left. \partial
_{s}q\right\vert _{\omega _{p}^{+}}\in H^{-\frac{1}{2}}(\omega _{p}^{+})%
\text{ and }\left. \partial _{s}q\right\vert _{\omega _{p}^{-}}\in H^{-\frac{%
1}{2}}(\omega _{p}^{-})\right\} .  \label{trace2}
\end{equation}
 \smallskip
In addition, if $q\in H^{0,0,2}(\Omega _{p})$, then the functional $\left.
\partial _{s}q\right\vert _{\omega _{p}^{+}}$ has the action, 
\[
\vartheta \rightarrow \int_{\omega _{p}^{+}}\partial _{s}q\vartheta d\omega
_{p}^{+}=\int_{\Omega _{p}}\partial _{s}^{2}q\vartheta d\Omega
_{p}+\int_{\Omega _{p}}q\partial _{s}\vartheta d\Omega _{p},\text{ \ for all 
}\vartheta \in H^{0,0,1}(\Omega _{p})
\]%
and likewise for $\left. \partial _{s}q\right\vert _{\omega _{p}^{+}}$. Thus, in a similar fashion
$
\left. \partial _{s}q\right\vert _{\omega _{p}^{+}},\left. \partial
_{s}q\right\vert _{\omega _{p}^{-}}\in \lbrack \gamma_0[H^{0,0,1}(\Omega _{p})]]'$

The above lemma allows us to handle traces on the anisotropic Sobolev spaces relevant to the plate pressure dynamics, and, in the typical way (e.g., \cite[Ch.2]{kesavan}) justify integration by parts in  $s$.

For the associated space of finite energy for the dynamics, we now set
\begin{align}
    \mathbf{V} &\equiv \{ \mathbf{u} \in L^2(\Omega_f) : \nabla \cdot \mathbf{u} = 0 \ \text{in} \ \Omega_f; ~~\mathbf{u} \cdot \mathbf{n} = 0 \ \text{on} \ \Gamma_f\}; \nonumber \\
    \mathbf{X_0} &\equiv \mathbf{H}^1_{\#,*}(\Omega_b) \times \mathbf{L}^2(\Omega_b) \times L^2(\Omega_b) \times H_0^2(\omega_p) \times L^2(\omega_p) \times L^2(\Omega_p) \times \mathbf{V}; \nonumber \\
    \mathbf{X} & \equiv \left\{ \left[ \mathbf{\eta },\zeta ,p_{b,}w,v,p_{p},\mathbf{u}\right]
     \in \mathbf{X_{0}} :\mathbf{\eta } =w\mathbf{e}_{3}\text{ on }\Gamma _{I}\right\}. \label{space}
\end{align}
We will topologize the energy space $\mathbf{X}$ with the  inner product: For $\mathbf{y}_{i}=\left[ \mathbf{\eta }_{i},\zeta
_{i},p_{b,i},w_{i},v_{i},p_{p,i},\mathbf{u}_{i}\right]^T$ in $\mathbf{X}$, 
\begin{align}
(\mathbf{y}_1, \mathbf{y}_2)_{X} &= a_{E}(\boldsymbol{\eta}_1, \boldsymbol{\eta}_2) + \rho_b (\boldsymbol{\zeta}_1, \boldsymbol{\zeta}_2)_{L^2(\Omega_b)} + c_b(p_{b,1}, p_{b,2})_{L^2(\Omega_b)} \nonumber \\ 
& + \rho_p (v_1, v_2)_{L^2(\omega_p)} + D( \Delta w_1, \Delta w_2)_{L^2(\omega_p)} + \gamma (w_1, w_2)_{L^2(\omega_p)}  + c_p (p_{p,1}, p_{p,2})_{L^2(\Omega_p)} \nonumber \\ 
&+ \rho_f (\mathbf{u}_1, \mathbf{u}_2)_{L^2(\Omega_f)}.  \label{energy}
\end{align}
\subsection{Operator and Domain} 
On the above energy space, we now proceed to define the multilayered Biot-Stokes dynamics operator $\mathbf{A}:\mathbf{X}\rightarrow \mathbf{X}$
  as follows:
\begin{definition}(Domain of $\mathbf{A}$)\label{def:domain}
    Let $\mathbf{y} = [\boldsymbol{\eta}, \boldsymbol{\zeta}, p_b, w, v, p_p, \mathbf{u}]^T \in \mathbf{X}$, then $\mathbf{y} \in \mathcal{D}(\mathbf{A})$ if and only if:  \newline {\bf (i)} The regularity conditions hold:
    \begin{itemize}
        \item [(A.1)] $\nabla \cdot \sigma^E(\boldsymbol{\eta}) \in \mathbf{L}^2(\Omega_b)$
        \item [(A.2)] $\boldsymbol{\zeta} \in \mathbf{H}^1_{\#,*}(\Omega_b)$
        \item [(A.3)] $p_b \in H^1_\#(\Omega_b)$ and $\sigma _{b}(\mathbf{\eta },p_{b})\mathbf{e}_{3}\in   \mathbf{H}_{\#}^{-\frac{1}{2}}( \Omega _{b})$)
         \item [(A.4)] $\Delta p_b \in L^2(\Omega_b)$ and $\nabla p_{b}\cdot \mathbf{n_f}\in H_{\#}^{-\frac{1}{2}}(\partial \Omega_{b})$
        \item [(A.5)] $v \in H_0^2(\omega_p)$
        \item [(A.6)] $p_p \in H^{0,0,2}(\Omega_p)$ (and so (\ref{trace}) and (\ref{trace2}) apply)
        \item [(A.7)] $\mathbf{u} \in \mathbf{H}_\#^1(\Omega_f)$.
    \end{itemize}
 {\bf (ii) }   The following pressure sub-problem regularity holds: $\exists \pi \in L^2(\Omega_f)$, such that
    \begin{itemize}
        \item [(A.8)] $\mu_f \Delta \mathbf{u} - \nabla \pi \in \mathbf{V}$ (and so this item and (A.7) yield $\sigma _{f}(\mathbf{u},\pi )\mathbf{e}_{3}\in \mathbf{H}^{-\frac{1}{2}%
         }(\partial \Omega _{f})$)
        \item [(A.9)] $2 \mu_f \mathbf{D}(\mathbf{u}) \vert_{ \partial \Omega_f} \in \mathbf{H}_\#^{-1/2}( \Omega_f)$
        \item [(A.10)] $\pi \vert_{\partial \Omega_f} \in H_\#^{-1/2}(\Omega_f)$
        \item [(A.11)] $\Delta \mathbf{u} \cdot \mathbf{n} \vert_{\Gamma_f} \in H^{-3/2}( \Gamma_f)$.
    \end{itemize}
  {\bf   (iii)} The additional additional regularity requirement, with respect to the plate component (well-defined because of (A.3) and (A.8)) holds:
    \begin{itemize}
        \item [(A.12)] $D\Delta _{\omega _{p}}^{2}w+\alpha _{p}\Delta _{\omega _{p}}\mathcal{K}%
(p_{p})+\left. \sigma _{f}(\mathbf{u},\pi )\mathbf{e}_{3}\cdot \mathbf{e}%
_{3}\right\vert _{\omega _{p}}-\left. \sigma _{b}(\mathbf{\eta },p_{b})%
\mathbf{e}_{3}\cdot \mathbf{e}_{3}\right\vert _{\omega _{p}}\in L^{2}(\omega
_{p})$.

    \end{itemize}
{\bf (iv)}   The following boundary interface conditions hold: 
    \begin{itemize}
        \item [(A.13)] $\langle 0, 0, v \rangle = \boldsymbol{\zeta} \vert_{\omega_p} \in \mathbf{H}^{1/2}(\Gamma_I)$
        \item [(A.14)] $\beta \mathbf{u}\vert_{\omega_p} \cdot \boldsymbol{\tau} = -\sigma _{f}(\mathbf{u},\pi )\mathbf{e}_{3}\cdot \boldsymbol{\tau} \vert_{\omega_p} \in H^{-1/2}(\Gamma_I)$
        \item [(A.15)] $k_p \partial_s p_p \vert_{\omega_p^+} = k_b \partial_{x_3} p_b \vert_{\omega_p} \in H^{-1/2}(\Gamma_I)$
        \item [(A.16)] $p_p\vert_{\omega_p^+} = p_b \vert_{\omega_p} \in H^{1/2}(\Gamma_I)$
        \item [(A.17)] $v - (\mathbf{u} \cdot \mathbf{e}_3)\vert_{\omega_p} = k_p \partial_s p_p \vert_{\omega_p^-} \in H^{-1/2}(\Gamma_I)$
        \item [(A.18)] $-\sigma _{f}(\mathbf{u},\pi )\mathbf{e}_{3} \cdot \mathbf{e}_3 \vert_{\omega_p} = p_p \vert_{\omega_p^-} \in H^{-1/2}(\Gamma_I)$.
    \end{itemize}
\end{definition}
\begin{remark}[Regularity of Traces] We note that in the matching interface conditions above, the identification of traces is done at the level of the most regular of the two, as inherited from prescribed properties of the energy space. All traces are defined at least distributionally on $\mathcal D(\mathbf A)$, and matched accordingly. Additional regularity on the domain can then be inferred.
\end{remark}

 Having specified the domain of $\mathbf{A}:\mathbf{X}\rightarrow \mathbf{X}$,
then  given $\mathbf{y} = [\boldsymbol{\eta}, \boldsymbol{\zeta}, p_b, w, v, p_p, \mathbf{u}]^T \in \mathcal{D}(\mathbf{A})$, we in turn prescribe its action:
\begin{equation}
\mathbf{A} \mathbf{y} = \begin{bmatrix}
    \boldsymbol{\zeta} \\
    \rho_b^{-1} \nabla \cdot \sigma^{E}(\boldsymbol{\eta}) - \rho_b^{-1} \alpha_b \nabla p_b \\
    - c_b^{-1} \alpha_b \nabla \cdot \boldsymbol{\zeta} + c_b^{-1} k_b \Delta p_b \\
    v \\
    \rho_p^{-1} (F_p - D \Delta_{\omega_p}^2 w - \gamma w - \alpha_p \Delta_{\omega_p} \mathcal{K}(p_p)) \\
    c_p^{-1} \alpha_p\tilde{\mathcal{K}}(\Delta v) + c_p^{-1} k_p \partial_{s}^2 p_p \\
    \rho_f^{-1}[\mu_f \Delta \mathbf{u} - \nabla \pi],
\end{bmatrix} \in \begin{bmatrix}
    \mathbf{H}^1_{\#,*}(\Omega_b) \\ \mathbf{L}^2(\Omega_b) \\ L^2(\Omega_b) \\ H_0^2(\omega_p) \\ L^2(\omega_p) \\ H^{0,0,1}(\Omega_p) \\ \mathbf{V}
\end{bmatrix}. \label{mat}
\end{equation}
Here, the difference term $F_p$ is as given in (\ref{flux}); recall
\begin{equation}
F_p=\left. \sigma _{b}(\mathbf{\eta },p_{b})\mathbf{e}_{3}\cdot \mathbf{e}%
_{3}\right\vert _{\omega _{p}}-\left. \sigma _{f}(\mathbf{u},\pi )\mathbf{e}%
_{3}\cdot \mathbf{e}_{3}\right\vert _{\omega _{p}}  \label{flux2}
\end{equation}
Moreover, $\pi = \Pi_1 p_p + \Pi_2 \mathbf{u} \in L^2(\Omega_f)$, where  $\Pi_i$ are as defined in (\ref{pf}); indeed, $\pi$, as the variable identified in (A.8) of the domain definition, is \textit{a fortiori} the solution to the fluid-pressure sub-problem (\ref{bvp}) \cite{AGW}. \\ \\
\indent In sum, the wellposedness question for the multilayered Biot-Stokes system (\ref{*.a})--(\ref{*.l}) becomes this abstract Cauchy problem: Given $\mathbb{F} = [0, \mathbf{F}_b, S, 0, 0, 0, \mathbf{F}_f]^T \in L^2(0,T;\mathbf{X})$, find $\mathbf{y} = [\boldsymbol{\eta}, \boldsymbol{\zeta}, p_b, w, v, p_p, \mathbf{u}]^T \in \mathcal{D}(\mathbf{A})$ such that
\begin{equation}
\frac{d}{dt}\mathbf{y}=\mathbf{Ay}+\mathbb{F}\text{, }~~\mathbf{y}(0)=\mathbf{y%
}_{0} \text{~ on~}(0,T) \text{.}  \label{cauchy}
\end{equation}

Accordingly, we proceed to establish strongly continuous contraction semigroup generation for\newline $\mathbf{A}:\mathcal{D}(\mathbf{A}) \subset \mathbf{X}\rightarrow \mathbf{X}$, to which end,  it is enough, by the Lumer-Phillips Theorem \cite{kesavan,pazy}, to establish its dissipativity (Section \ref{dissss}) and maximality (Section \ref{maxx}). The existence of strong and mild (semigroup) solutions follow then for the Cauchy problem in the standard fashion by direct invocation of Lumer-Phillips.

\section{Precise Statement of Main Results}

Our central supporting theorem is that of semigroup generation for  $\cA$. 
\begin{theorem}
\label{th:main1}
    Let parameters $c_b, c_p  > 0$ in the Biot-Plate-Stokes system (\ref{*.a})--(\ref{*.l}). Then the associated modeling operator $\mathcal A$ on $X$,defined by \eqref{mat}, with domain $\mathcal D(\mathcal A)$ given in Definition \ref{def:domain}, is the generator of a strongly continuous semigroup $\{e^{\cA t}: t\geq 0\}$ of contractions on $X$. Thus, for $\mathbf y_0 \in \mathcal D(\mathcal A)$, we have $e^{\mathcal A\cdot}\mathbf y_0 \in C([0,T];\mathcal D(\mathcal A))\cap C^1((0,T);X)$ satisfying \eqref{cauchy} in the {\em strong sense} with $\mathcal F=[\mathbf 0,\mathbf 0, 0, 0, 0, 0, \mathbf 0]^T$; similarly, for $\mathbf y_0 \in X$, we have $e^{\mathcal A\cdot}\mathbf y_0 \in C([0,T];X)$ satisfying \eqref{cauchy} in the {\em generalized sense} with $\mathcal F=[\mathbf 0, \mathbf 0, 0, 0, 0,\mathbf 0]^T$.
\end{theorem}

As a corollary, we  obtain strong and mild (or generalized) solutions to \eqref{*.a}--\eqref{*.l}, including  in the non-homogeneous case with generic forces. By strong solutions to  \eqref{*.a}--\eqref{*.l} we mean  \eqref{*.a}--\eqref{*.l}  hold  point-wise in time; mild solutions satisfy the time integrated (variation of parameters form) of  \eqref{*.a}--\eqref{*.l}. Through limiting procedures (as in \cite{AW} following \cite{AGW})---using smooth data for the semigroup when $c_b, c_p>0$, and subsequently using a ``vanishing compressibility" argument as $c_b, c_p \searrow 0$---we can obtain weak solutions for all $c_b, c_p \ge 0$. Furthermore as in \cite{AGW,AW} these weak solutions are well-posed.

Finally we address well-posedness of the system in the presence of the von K\'arm\'an plate non-linearity. This will be shown to be a perturbation of the Cauchy problem (\ref{cauchy}) with a locally Lipschitz forcing function $\mathcal{F}(\mathbf{y})$.
\begin{equation}
\frac{d}{dt}\mathbf{y}=\mathbf{Ay} + \mathcal{F}(\mathbf{y})\text{, }\mathbf{y}(0)=\mathbf{y%
}_{0} \text{~ on~}(0,T) \text{.}  \label{cauchy_vk}
\end{equation}
From this perturbation local well-posedness is obtained, and global (in time) well-posedness follows from an \emph{a priori} energy identity that ensures solutions are bounded in the energy norm. 

\begin{theorem}\label{th:main4}
For any $T>0$, the perturbed Cauchy problem (\ref{cauchy_vk}) is semigroup well-posed on $\mathbf{X}$ for all $[0,T]$. That is to say
that the PDE problem in \eqref{*.a}--\eqref{*.l}, taking into account
the nonlinear plate equation \eqref{nlp}, is well-posed on $[0,T]$ in the sense of finite-energy
mild solutions.

Moreover, we have the \emph{%
global-in-time} estimate for solutions: For any $T>0$
\begin{equation}  \label{globalbound}
\sup_{t \in [0,T]} \mathcal{E}(\bm{\eta}(t), \bm{\zeta}(t),p_b(t),w_{t}(t), v(t), p_p(t), \mathbf{u}(t)) \le 
\mathbf{C}(\bm{\eta}_0, \bm{\zeta}_0, p_{b,0}, w_0, v_0, p_{p,0}, \mathbf{u}_0, F_0).
\end{equation}
\end{theorem}

\section{Dissipativity of the Operator \texorpdfstring{$\mathbf{A}:\mathbf{X} \rightarrow \mathbf{X}$}{}}\label{dissss}

Let $\mathbf{y} = [\boldsymbol{\eta}, \boldsymbol{\zeta}, p_b, w, v, p_p, \mathbf{u} ]^T \in \mathcal{D}(\mathbf{A})$ be given. Then, via the description of the operator's action in (\ref{mat}) we have,
\begin{align}
    (\mathbf{A} y, y)_\mathbf{X} &= \left(
    \begin{bmatrix}
    \boldsymbol{\zeta} \\
    \rho_b^{-1} \nabla \cdot \sigma^{E}(\boldsymbol{\eta}) - \rho_b^{-1} \alpha_b \nabla p_b \\
    - c_b^{-1} \alpha_b \nabla \cdot \boldsymbol{\zeta} + c_b^{-1} k_b \Delta p_b \\
    v \\
    \rho_p^{-1} (F_p - D \Delta_{\omega_p}^2 w - \gamma w - \alpha_p \Delta_{\omega_p} \mathcal{K}(p_p)) \\
    c_p^{-1} \alpha_p\tilde{\mathcal{K}}(\Delta v) + c_p^{-1} k_p \partial_{s}^2 p_p \\
    \rho_f^{-1}[\mu_f \Delta \mathbf{u} - \nabla \pi] \end{bmatrix},
    \begin{bmatrix}
    \boldsymbol{\eta} \\ \boldsymbol{\zeta} \\ p_b \\ w \\ v \\ p_p \\ \mathbf{u}
    \end{bmatrix}
    \right)_{X}  \label{d1}
    \end{align}
 Proceeding as we did in the computations for the energy identity above, and
taking into account the definition of the energy inner product in (\ref{energy}), this computation becomes

\begin{eqnarray}
\left( \mathbf{Ay},\mathbf{y}\right) _{\mathbf{X}} &=&-k_{p}\left\Vert \nabla
p_{b}\right\Vert _{L^{2}(\Omega _{b})}^{2}-2\mu _{f}\left\Vert \mathbf{D}(%
\mathbf{u)}\right\Vert _{L^{2}(\Omega _{f})}^{2}-k_{p}\left\Vert \partial
_{s}p_{p}\right\Vert _{L^{2}(\omega _{p})}^{2}  \nonumber \\
&&+I_{bdry}+\left\langle F_p,v\right\rangle _{\omega _{p}},  \label{d2}
\end{eqnarray}
  
where $I_{bdry}$ term is,
\begin{equation}
I_{bdry} = - \int_{\Gamma_I} \sigma_b \mathbf{e}_3 \cdot \boldsymbol{\zeta} d\Gamma_I -\left\langle \nabla p_{b}\cdot \mathbf{e}_{3},p_{b}\right\rangle _{\omega
_{p}}+k_{p}\left\langle \partial _{s}p_{p},p_{p}(\mathbf{n}%
_{p})_{s}\right\rangle _{\partial \Omega _{p}} + \int_{\Gamma_I} \sigma_f \mathbf{e}_3 \cdot \mathbf{u} d\Gamma_I.
\label{d3}
\end{equation}
Again, $\mathbf{n}_{p}=[n_{1},n_{2},n_{s}]$ is
the exterior unit normal vector on $\Omega _{p}$. Also, the difference of fluxes term $F_p$ is as given in (\ref{flux2}). (In this work, we are also using the fact that $\mathcal{K}\in \mathcal{L}%
\left( L^{2}(\Omega _{p}),L^{2}(\omega _{p})\right) $ and $\mathcal{\tilde{K}%
}\in \mathcal{L}\left( L^{2}(\omega _{p}),L^{2}(\Omega _{p})\right) $ are
adjoint to each other.)

Using the domain boundary interface conditions (A.14), (A.17) and (A.18), we have that
\begin{align*}
    \int_{\Gamma_I} \sigma_f \mathbf{e}_3 \cdot \mathbf{u} d\Gamma_I - \int_{\Gamma_I} \sigma_f \mathbf{e}_3 \cdot \mathbf{e}_3 v d\Gamma_I &= \int_{\Gamma_I} \sigma_f \mathbf{e}_3 \cdot \mathbf{e}_i (\mathbf{u} \cdot \mathbf{e}_i) d\Gamma_I - \int_{\Gamma_I} \sigma_f \mathbf{e}_3 \cdot \mathbf{e}_3 v d\Gamma_I\\
    &= \int_{\Gamma_I} \sigma_f \mathbf{e}_3 \cdot \boldsymbol{\tau} (\mathbf{u} \cdot \boldsymbol{\tau})d\Gamma_I + \int_{\Gamma_I} \sigma_f \mathbf{e}_3 \cdot \mathbf{e}_3 [(\mathbf{u} \cdot \mathbf{e}_3) - v] d\Gamma_I\\
    &= -\beta \Vert \mathbf{u} \cdot \boldsymbol{\tau} \Vert^2_{L^2(\omega_p)} + \int_{\omega_p^-} k_p (\partial_s p_p) p_p d\omega_p^-.
\end{align*}

This relation in turns gives%
\begin{eqnarray}
&&I_{bdry}+ \left\langle F_p,v\right\rangle _{\omega _{p}} \nonumber \\
&=&- \cancel{\int_{\Gamma _{i}}\sigma _{b}\mathbf{e}_{3}\cdot \zeta d\Gamma
_{I}} - \bcancel{\left\langle \nabla p_{b}\cdot \mathbf{e}_{3},p_{b}\right\rangle
_{\omega _{p}}} + \bcancel{k_{p}\left\langle \partial _{s}p_{p},p_{b}\right\rangle
_{\omega _{p}^{-}}} \nonumber \\
&&-\beta \left\Vert \mathbf{u}\cdot \mathbf{\tau }\right\Vert _{\Gamma
_{I}}^{2}+\cancel{\left\langle \sigma _{b}\mathbf{e}_{3}\cdot \mathbf{e}%
_{3},v\right\rangle _{\omega _{p}}} \nonumber \\
&=& -\beta \left\Vert \mathbf{u}\cdot \mathbf{\tau }\right\Vert _{\Gamma
_{I}}^{2}, \label{d5}
\end{eqnarray}%
 where in the last step, we used the domain boundary interface conditions (A.13), (A.15) and (A.16).

Finally, applying the relation (\ref{d5}) to the RHS of (\ref{d3}), we have the relation, $\forall \mathbf{y} \in \mathcal{D}(\mathbf{A})$,
$$(\mathbf{A}\mathbf{y},\mathbf{y})_X = -k_b \Vert \nabla p_b \Vert^2_{L^2(\Omega_b)} - 2\mu_f \Vert \mathbf{D}(\mathbf{u})\Vert^2_{L^2(\Omega_f)} - k_p \Vert \partial_s p_p \Vert^2_{L^2(\Omega_p)} -\beta \Vert \mathbf{u} \cdot \boldsymbol{\tau} \Vert^2_{L^2(\omega_p)} \leq 0.$$
That is, $\mathbf{A}$ is dissipative on $\mathbf{X}$.

\section{Maximality}\label{maxx}
\subsection{Derivation of Associated Bilinear Forms}
Given $\mathbf{F} \in \mathbf{X}$, and given parameter $\lambda > 0$, we wish to find  $\mathbf{y} = [\boldsymbol{\eta}, \boldsymbol{\zeta}, p_b, w, v, p_p, \mathbf{u}]^T \in \mathcal{D}(\mathbf{A})$ which solves the equation
\begin{equation}
(\lambda I - \mathbf{A}) \mathbf{y} = \mathbf{F}.  \label{resolve}
\end{equation}
Choosing $\lambda = 1$ in (\ref{resolve}), this resolvent relation becomes:
\begin{equation}
\label{full-form}
\begin{bmatrix}
\boldsymbol{\eta} - \boldsymbol{\zeta} \\
\boldsymbol{\zeta} - \rho_b^{-1} \nabla \cdot \sigma^{E}(\boldsymbol{\eta}) + \alpha_b \rho_b^{-1} \nabla p_b \\
p_b + c_b^{-1} \alpha_b \nabla \cdot \boldsymbol{\zeta} - c_b^{-1} k_b \Delta p_b \\
w - v \\
v + \rho_p^{-1} [D \Delta_{\omega_p}^2 w + \gamma w + \alpha_p \Delta_{\omega_p} \mathcal{K}(p_p) - F_p] \\
p_p - \alpha_p c_p^{-1} \tilde{\mathcal{K}}(\Delta_{\omega_p} v) - c_p^{-1} k_p \partial_s^2 p_p \\
\mathbf{u} - \rho_f^{-1} \mu_f \Delta \mathbf{u} + \rho_f^{-1} \nabla \pi \\
\end{bmatrix} =
\begin{bmatrix}
\mathbf{f}_1 \\ \mathbf{f}_2 \\ f_3 \\ f_4 \\ f_5 \\ f_6 \\ \mathbf{f}_7
\end{bmatrix} \in \mathbf{X}.
\end{equation}
Via the substitutions $\boldsymbol{\zeta} = \boldsymbol{\eta} - \mathbf{f}_1$ and $v = w - f_4$, we then arrive at, 
\begin{equation} \label{eg:reduced-form}
\begin{bmatrix}
\boldsymbol{\eta} - \rho_b^{-1} \nabla \cdot \sigma^{E}(\boldsymbol{\eta}) + \alpha_b \rho_b^{-1} \nabla p_b \\
p_b + c_b^{-1} \alpha_b \nabla \cdot \boldsymbol{\eta} - c_b^{-1} k_b \Delta p_b \\
w + \rho_p^{-1} [D \Delta_{\omega_p}^2 w + \gamma w + \alpha_p \Delta_{\omega_p} \mathcal{K}(p_p) - F_p] \\
p_p - \alpha_p c_p^{-1} \tilde{\mathcal{K}}(\Delta_{\omega_p} w) - c_p^{-1} k_p \partial_s^2 p_p \\
\mathbf{u} - \rho_f^{-1} \mu_f \Delta \mathbf{u} + \rho_f^{-1} \nabla \pi \\
\end{bmatrix} =
\begin{bmatrix}
\mathbf{f}_2 + \mathbf{f}_1 \\
f_3 + c_b^{-1} \alpha_b \nabla \cdot \mathbf{f}_1 \\
f_5 + f_4 \\
f_6 - \alpha_p c_p^{-1} \tilde{\mathcal{K}}(\Delta_{\omega_p} f_4) \\
\mathbf{f}_7
\end{bmatrix}.
\end{equation}
By way of dealing with this relation, we define the auxiliary space
\begin{multline*}
\Sigma = \Bigl\{ \boldsymbol{\psi} = [\boldsymbol{k}, q_b, \xi, q_p, \boldsymbol{\delta}]^T \in \mathbf{H}^1_{\#,*}(\Omega_b) \times H^1_{\#,*}(\Omega_b) \times H^2_{0}(\omega_p) \times H^{0,0,1}(\Omega_p) \times \mathbf{H}^1_{\#, *}(\Omega_f): \\ 
\boldsymbol{k} \vert_{\Gamma_I} = \xi \mathbf{e}_3, \ q_b \vert_{\Gamma_I} = q_p \vert_{\omega_p^+} \Bigr\},
\end{multline*}
(Note that we are not imposing the divergence free constraint on the fluid component of this space.) We equip $\Sigma$ with the following inner product: for given $\boldsymbol{\phi} = [\boldsymbol{\eta}, p_b, w, p_p, \mathbf{u}]^T \in \Sigma$ and $\boldsymbol{\psi} = [\mathbf{k}, q_b, \xi, q_p, \boldsymbol{\delta}]^T \in \Sigma$, 
$$(\boldsymbol{\phi}, \boldsymbol{\psi})_{\Sigma} = a_E(\boldsymbol{\eta}, \mathbf{k}) + (\nabla p_b, \nabla q_b)_{L^2(\Omega_b)} + D( \Delta w, \Delta \xi)_{L^2(\omega_p)} + (p_p, q_p)_{H^{0,0,1}(\Omega_p)} + \mu_f (\nabla \mathbf{u}, \nabla \boldsymbol{\delta})_{L^2(\Omega_f)}.$$
We can now define the weak form of our reduced problem. Supposing that $\mathbf{y} = [\boldsymbol{\eta}, \boldsymbol{\zeta}, p_b, w, v, p_p, \mathbf{u}]^T \in \mathcal{D}(\mathbf{A})$ is a soluion of (\ref{resolve}), we take the (initially formal) $\Sigma$-inner product of our reduced system (\ref{eg:reduced-form}), with respect to $\mathbf{\psi }=[\mathbf{k},q_{b},\xi ,q_{p},\mathbf{\delta }]\in \Sigma $
:
\begin{equation}
\left(
\begin{bmatrix}
\boldsymbol{\eta} - \rho_b^{-1} \nabla \cdot \sigma^E(\boldsymbol{\eta}) + \alpha_b \rho_b^{-1} \nabla p_b \\
p_b + c_b^{-1} \alpha_b \nabla \cdot \boldsymbol{\eta} - c_b^{-1} k_b \Delta p_b \\
w + \rho_p^{-1} [D \Delta_{\omega_p}^2 w + \gamma w + \alpha_p \Delta_{\omega_p} \mathcal{K}(p_p) - F_p] \\
p_p - \alpha_p c_p^{-1} \tilde{\mathcal{K}}(\Delta_{\omega_p} w) - c_p^{-1} k_p \partial_s^2 p_p \\
\mathbf{u} - \rho_f^{-1} \mu_f \Delta \mathbf{u} + \rho_f^{-1} \nabla \pi \\
\end{bmatrix}
,
\boldsymbol{\psi}
\right)_{\Sigma}
=
\left(
\begin{bmatrix}
\mathbf{f}_2 + \mathbf{f}_1 \\
f_3 + c_b^{-1} \alpha_b \nabla \cdot \mathbf{f}_1 \\
f_5 + f_4 \\
f_6 - \alpha_p c_p^{-1} \tilde{\mathcal{K}}(\Delta_{\omega_p} f_4) \\
\mathbf{f}_7
\end{bmatrix}
,
\boldsymbol{\psi}
\right)_{\Sigma}. \label{reduced}
\end{equation}
We proceed to simplify our system of equations using formal calculations, line by line:

\begin{align*}
(R.1) ~~~\rho_b (\boldsymbol{\eta}, \mathbf{k})_{L^2(\Omega_b)} &- (\nabla \cdot \sigma^{E}(\boldsymbol{\eta}), \mathbf{k})_{L^2(\Omega_b)} + \alpha_b (\nabla p_b, \mathbf{k})_{L^2(\Omega_b)} \\
&= \rho_b (\boldsymbol{\eta}, \mathbf{k})_{L^2(\Omega_b)} + (\sigma^{E} (\boldsymbol{\eta}), \nabla \mathbf{k})_{L^2(\Omega_b)} - \alpha_b (p_b, \nabla \cdot \mathbf{k})_{L^2(\Omega_b)} - \int_{\partial \Omega_b} \sigma_b \cdot \mathbf{n} \mathbf{k} \\
&= \rho_b (\boldsymbol{\eta}, \mathbf{k})_{L^2(\Omega_b)} + a_{E} (\boldsymbol{\eta}, \mathbf{k}) - \alpha_b (p_b, \nabla \cdot \mathbf{k})_{L^2(\Omega_b)} + \int_{\Gamma_I} \sigma_b \cdot \mathbf{e}_3 \mathbf{k} .
\end{align*}
\begin{align*}
(R.2) ~~~ c_b (p_b, q_b)_{L^2(\Omega_b)} &+ \alpha_b (\nabla \cdot \boldsymbol{\eta}, q_b)_{L^2(\Omega_b)} - k_b (\Delta p_b, q_b)_{L^2(\Omega_b)} \\
&= c_b (p_b, q_b)_{L^2(\Omega_b)} + \alpha_b (\nabla \cdot \boldsymbol{\eta}, q_b)_{L^2(\Omega_b)} + k_b (\nabla p_b, \nabla q_b)_{L^2(\Omega_b)} - k_b \int_{\partial \Omega_b} \nabla p_b \cdot \mathbf{n} q_b \\
&= c_b (p_b, q_b)_{L^2(\Omega_b)} + \alpha_b (\nabla \cdot \boldsymbol{\eta}, q_b)_{L^2(\Omega_b)} + k_b (\nabla p_b, \nabla q_b)_{L^2(\Omega_b)} + k_b \int_{\Gamma_I} \nabla p_b \cdot \mathbf{e}_3 q_b.
\end{align*}
\begin{align*}
(R.3)~~~ \rho_p (w, \xi)_{L^2(\omega_p)} &+ ([D \Delta_{\omega_p}^2 w + \gamma w + \alpha_p \Delta_{\omega_p} \mathcal{K}(p_p) - F_p], \xi)_{L^2(\omega_p)} \\
&= \rho_p (w, \xi)_{L^2(\omega_p)} + (D \Delta_{\omega_p} w + \alpha_p \mathcal{K}(p_p), \Delta_{\omega_p} \xi)_{L^2(\omega_p)} + \gamma (w, \xi)_{L^2(\omega_p)} - (F_p, \xi)_{L^2(\omega_p)}
\end{align*}
\begin{align*}
(R.4) ~~~ c_p (p_p, q_p)_{L^2(\Omega_p)} &- \alpha_p (\tilde{\mathcal{K}}(\Delta_{\omega_p} w), q_p)_{L^2(\Omega_p)} - k_p (\partial_s^2 p_p, q_p)_{L^2(\Omega_p)} \\
&= c_p (p_p, q_p)_{L^2(\Omega_p)} - \alpha_p (\tilde{\mathcal{K}}(\Delta_{\omega_p} w), q_p)_{L^2(\Omega_p)} + k_p (\partial_s p_p, \partial_s q_p)_{L^2(\Omega_p)} \\
   & ~~~~- \int_{\partial \Omega_p} k_p \partial_s p_p q_p (\mathbf{n_{p}})_s d \partial \Omega_p \\
&= c_p (p_p, q_p)_{L^2(\Omega_p)} - \alpha_p (\tilde{\mathcal{K}}(\Delta_{\omega_p} w), q_p)_{L^2(\Omega_p)} + k_p (\partial_s p_p, \partial_s q_p)_{L^2(\Omega_p)} \\
&- \int_{\omega_p^+} k_p [\partial_s p_p] q_p d \omega_p^++ \int_{\omega_p^-} k_p [\partial_s p_p] q_p d \omega_p^-
\end{align*}
\begin{align*}
(R.5) ~~~ \rho_f (\mathbf{u}, \boldsymbol{\delta})_{L^2(\Omega_f)} &- \mu_f (\Delta \mathbf{u}, \boldsymbol{\delta})_{L^2(\Omega_f)} + (\nabla \pi, \boldsymbol{\delta})_{L^2(\Omega_f)} \\
&= \rho_f (\mathbf{u}, \boldsymbol{\delta})_{L^2(\Omega_f)} + \mu_f (\nabla \mathbf{u}, \nabla \boldsymbol{\delta})_{L^2(\Omega_f)} - \int_{\partial \Omega_f} \nabla \mathbf{u} \cdot \mathbf{n} \boldsymbol{\delta} - (\pi, \nabla \cdot \boldsymbol{\delta})_{L^2(\Omega_f)} + \int_{\partial \Omega_f} \pi \boldsymbol{\delta} \cdot \mathbf{n} \\
&= \rho_f (\mathbf{u}, \boldsymbol{\delta})_{L^2(\Omega_f)} + 2 \mu_f (D(\mathbf{u}), D(\boldsymbol{\delta}))_{L^2(\Omega_f)} - (\pi, \nabla \cdot \boldsymbol{\delta})_{L^2(\Omega_f)} - \int_{\partial \Omega_f} \sigma_f \cdot \mathbf{n} \boldsymbol{\delta} \\
&= \rho_f (\mathbf{u}, \boldsymbol{\delta})_{L^2(\Omega_f)} + 2 \mu_f (D(\mathbf{u}), D(\boldsymbol{\delta}))_{L^2(\Omega_f)} - (\pi, \nabla \cdot \boldsymbol{\delta})_{L^2(\Omega_f)} - \int_{\Gamma_I} \sigma_f \cdot \mathbf{e}_3 \boldsymbol{\delta}.
\end{align*}
We moreover expand $F_p$ into its constituent parts,
\begin{equation}
- \int_{\Gamma_I} F_p \xi = - \int_{\Gamma_I} \sigma_b \mathbf{e}_3 \cdot \mathbf{e}_3 \xi + \int_{\Gamma_I} \sigma_f \mathbf{e}_3 \cdot \mathbf{e}_3 \xi. \label{r6}
\end{equation}

\medskip 
Now, concerning the terms above that contain $\sigma_f$, we have that
\begin{align}
\int_{\Gamma_I} \sigma_f \mathbf{e}_3 \cdot \mathbf{e}_3 \xi &- \int_{\Gamma_I} \sigma_f \cdot \mathbf{e}_3 \boldsymbol{\delta} \nonumber \\
&= \int_{\Gamma_I} \sigma_f \mathbf{e}_3 \cdot \mathbf{e}_3 [\xi - (\boldsymbol{\delta} \cdot \mathbf{e}_3)] - \int_{\Gamma_I} \sigma_f \mathbf{e}_3 \cdot \boldsymbol{\tau} (\boldsymbol{\delta} \cdot \boldsymbol{\tau}) \nonumber \\
&= -\int_{\Gamma _{I}}p_{p}(\cdot ,-h/2) [\xi - (\boldsymbol{\delta} \cdot \mathbf{e}_3)] + \beta \int_{\Gamma_I} (\mathbf{u} \cdot \boldsymbol{\tau}) (\boldsymbol{\delta} \cdot \boldsymbol{\tau}),  \label{i7}
\end{align}
after using the boundary interface conditions (A.14) and (A.18). Similarly concerning the terms above which contain $\sigma_b$, we have that
\begin{align}
- \int_{\Gamma_I} \sigma_b \mathbf{e}_3 \cdot \mathbf{e}_3 \xi  &+ \int_{\Gamma_I} \sigma_b \mathbf{e}_3 \cdot \mathbf{k} =0, \label{i8}
\end{align}
after using the interface condition imposed on the test space. 
At this point, we add $(R.1)$, $(R.3)$ and $(R.5)$, taking into account (\ref%
{r6})-(\ref{i8}). We have then,%
\begin{equation}
\begin{array}{l}
(R.1)+(R.3)+(R.5)=\rho _{b}\left( \mathbf{\eta },\mathbf{k}\right)
_{L^{2}(\Omega _{b})}+a_{E}\left( \mathbf{\eta },\mathbf{k}\right)  \\ 
\text{ \ }+\rho _{p}\left( w,\xi \right) _{L^{2}(\Omega _{p})}+D\left(
\Delta_{\omega_p} w,\Delta_{\omega_p} \xi \right) _{L^{2}(\Omega _{p})}+\gamma \left( w,\xi
\right) _{L^{2}(\Omega _{p})} \\ 
\text{ \ }+\rho _{f}\left( \mathbf{u},\mathbf{\delta }\right) _{L^{2}(\Omega
_{f})}+2\mu _{f}\left( \mathbf{D)} (\mathbf{u}),\mathbf{D} (\mathbf{\delta )}\right)
_{L^{2}(\Omega _{f})}-\int_{\Gamma _{I}}p_{p}(\cdot ,-h/2)\left[ \xi -%
\mathbf{\delta \cdot e}_{3}\right] +\beta \int_{\Gamma _{I}}\left( \mathbf{u}%
\cdot \mathbf{\tau }\right) \left( \mathbf{\delta }\cdot \mathbf{\tau }%
\right)  \\ 
\text{ \ \ }-\alpha _{b}\left( p_{b},\nabla \cdot \mathbf{k}\right)
_{L^{2}(\Omega _{b})}+\alpha _{p}\left( \mathcal{K(}p_{p}),\Delta \xi
\right) _{L^{2}(\Omega _{p})}-\left( \pi ,\nabla \cdot \mathbf{\delta }%
\right) _{L^{2}(\Omega _{f})}.%
\end{array}
\label{inter}
\end{equation}

Furthermore:  recalling that $f_4 = w - v$, we then have from domain criterion (A.17) that
\begin{align}
\int_{\omega_p^-} k_p [\partial_s p_p] q_p &= \int_{\Gamma_I} [v - (\mathbf{u} \cdot \mathbf{e}_3)] q_p(\cdot, -h/2) + \int_{\Gamma_I} f_4 q_p(\cdot, -h/2) - \int_{\Gamma_I} f_4 q_p(\cdot, -h/2) \nonumber \\
&= \int_{\Gamma_I} [w - (\mathbf{u} \cdot \mathbf{e}_3)] q_p(\cdot, -h/2) - \int_{\Gamma_I} f_4 q_p(\cdot, -h/2). \label{i9} 
\end{align}
By (A.18), we also have,

(The $\int_{\omega_p} f_4 q_p(\cdot, -h)$ term will be added to the forcing term of the pending inf-sup formulation.)
In addition, by (A.15) and (A.16),
\begin{equation}
\int_{\Gamma_I} k_b \nabla p_b \cdot \mathbf{e}_3 q_b - \int_{\omega_p^+} k_p [\partial_s p_p] q_p = 0.
\label{i10}
\end{equation}

Thus, from (\ref{reduced}) and (\ref{r6})-(\ref{i10}): if $\mathbf{y} = [\boldsymbol{\eta}, \boldsymbol{\zeta}, p_b, w, v, p_p, \mathbf{u}]^T \in \mathcal{D}(\mathbf{A})$ solves the resolvent equation (\ref{resolve}), then for all $\mathbf{\psi }=[\mathbf{k},q_{b},\xi ,q_{p},\boldsymbol{\delta }]\in \Sigma $,

\begin{equation}
\begin{array}{l}
\left( \mathbf{F},\mathbf{\psi }\right) _{\mathbf{X}}=\left( \left( I-%
\mathbf{A}\right) \mathbf{y},\mathbf{\psi }\right) _{\mathbf{X}%
}=\sum_{i=1}^{5}(R.i) \\ 
=\rho _{b}\left( \mathbf{\eta },\mathbf{k}\right) _{L^{2}(\Omega
_{b})}+a_{E}\left( \mathbf{\eta },\mathbf{k}\right) +c_{b}\left(
p_{b},q_{b}\right) _{L^{2}(\Omega _{b})}+k_{b}\left( \nabla p_{b},\nabla
q_{b}\right) _{L^{2}(\Omega _{b})} \\ 
\text{ \ }+\rho _{p}\left( w,\xi \right) _{L^{2}(\Omega _{p})}+D\left(
\Delta_{\omega_p} w,\Delta_{\omega_p} \xi \right) _{L^{2}(\Omega _{p})}+\gamma \left( w,\xi
\right) _{L^{2}(\Omega _{p})} \\ 
\text{ \ \ }+c_{p}\left( p_{p},q_{p}\right) _{L^{2}(\Omega
_{p})}+k_{p}\left( \partial _{s}p_{p},\partial _{s}q_{p}\right)
_{L^{2}(\Omega _{p})} \\ 
\text{ \ }+\rho _{f}\left( \mathbf{u},\boldsymbol{\delta }\right) _{L^{2}(\Omega
_{f})}+2\mu _{f}\left( \mathbf{D}(\mathbf{u}),\mathbf{D}(\boldsymbol{\delta }%
)\right) _{L^{2}(\Omega _{f})}-\int_{\Gamma _{I}}p_{p}(\cdot ,-h/2)\left[
\xi -\boldsymbol{\delta \cdot e}_{3}\right] +\int_{\Gamma _{I}}\left[ w-\mathbf{%
u\cdot e}_{3}\right] q_{p}(\cdot ,-h/2) \\ 
\text{ \ \ }+\beta \int_{\Gamma _{I}}\left( \mathbf{u}\cdot \mathbf{\tau }%
\right) \left(\boldsymbol{\delta }\cdot \mathbf{\tau }\right) -\alpha
_{b}\left( p_{b},\nabla \cdot \mathbf{k}\right) _{L^{2}(\Omega _{b})}+\alpha
_{b}\left( \nabla \cdot \boldsymbol{\eta },q_{b},\right) _{L^{2}(\Omega
_{b})}+\alpha _{p}\left( \mathcal{K}\left( p_{p}\right) ,\Delta _{\omega
_{p}}\xi \right) _{_{L^{2}(\omega _{p})}} \\ 
\text{ \ \ }-\alpha _{p}\left( \widetilde{\mathcal{K}}\left( \Delta _{\omega
_{p}}w\right) ,q_{p}\right) _{_{L^{2}(\Omega _{p})}}-\left( \pi ,\nabla
\cdot \boldsymbol{\delta }\right) _{L^{2}(\Omega _{f})}-\int_{\Gamma
_{I}}f_{4}q_{p}(\cdot ,-h/2) .%
\end{array}
\label{var}
\end{equation}

The characterization of the resolvent equation solution leads us to the following variational problem: Find a pair [$\boldsymbol{\phi}, \pi] \in \Sigma \times L^2(\Omega_f)$ such that 
\begin{alignat}{2}
a(\boldsymbol{\phi}, \boldsymbol{\psi}) + b(\boldsymbol{\psi}, \pi) &= \mathcal{f}(\boldsymbol{\psi}) \quad &&\forall \boldsymbol{\psi} \in \Sigma \label{rel} \\
b(\boldsymbol{\psi}, \chi) &= 0 \quad &&\forall \chi \in L^2(\Omega_f), \label{con}
\end{alignat}

Here (taken from (\ref{var}), the bilinear forms and data are as follows: for given $\boldsymbol{\phi} = [\boldsymbol{\eta}, p_b, w, p_p, \mathbf{u}]^T \in \Sigma$ and $\boldsymbol{\psi} = [\mathbf{k}, q_b, \xi, q_p, \boldsymbol{\delta}]^T \in \Sigma$, and $\chi \in L^2(\Omega_f)$
\begin{equation} \label{eq:bilinear_a}
\begin{aligned}
a(\boldsymbol{\phi}, \boldsymbol{\psi}) \equiv & ~\rho_b (\boldsymbol{\eta}, \mathbf{k})_{L^2(\Omega_b)} + a_{E}(\boldsymbol{\eta}, \mathbf{k}) + c_b (p_b, q_b)_{L^2(\Omega_b)} + k_b (\nabla p_b, \nabla q_b)_{L^2(\Omega_b)} \\
&+ \rho_p (w, \xi)_{L^2(\omega_p)} + D(\Delta_{\omega_p} w, \Delta_{\omega_p} \xi)_{L^2(\omega_p)} + \gamma (w, \xi)_{L^2(\omega_p)} + c_p (p_p, q_p)_{L^2(\Omega_p)} \\
&+ k_p (\partial_s p_p, \partial_s q_p)_{L^2(\Omega_p)} + \rho_f (\mathbf{u}, \boldsymbol{\delta})_{L^2(\Omega_f)} + 2 \mu_f (\mathbf{D}(\mathbf{u}), \mathbf{D}(\boldsymbol{\delta}))_{L^2(\Omega_f)} \\
&+ \beta \int_{\Gamma_I} (\mathbf{u} \cdot \boldsymbol{\tau}) (\boldsymbol{\delta} \cdot \boldsymbol{\tau}) + \alpha_p (\mathcal{K}(p_p), \Delta_{\omega_p} \xi)_{L^2(\omega_p)} - \alpha_p (\tilde{\mathcal{K}}(\Delta_{\omega_p} w), q_p)_{L^2(\omega_p)} \\
&+ \alpha_b (\nabla \cdot \boldsymbol{\eta}, q_b)_{L^2(\Omega_b)} - \alpha_b (p_b, \nabla \cdot \mathbf{k})_{L^2(\Omega_b)} \\
&+ \int_{\Gamma_I} [w - (\mathbf{u} \cdot \mathbf{e}_3)] q_p(\cdot,-h/2) - \int_{\Gamma_I} p_p(\cdot,-h/2) [\xi - (\boldsymbol{\delta} \cdot \mathbf{e}_3)];
\end{aligned}
\end{equation}
\begin{equation} \label{eq:bilinear_b}
\hspace*{-10cm} b(\boldsymbol{\psi}, \chi) \equiv - (\chi, \nabla \cdot \boldsymbol{\delta})_{L^2(\Omega_f)};
\end{equation}
\begin{equation}
\hspace*{-3.2cm}
\begin{aligned}
\mathcal{f}(\boldsymbol{\psi}) \equiv & ~ \rho_b (\mathbf{f}_2 + \mathbf{f}_1, \mathbf{k})_{L^2(\Omega_b)} + (c_b f_3 + \alpha_b \nabla \cdot \mathbf{f}_1, q_b)_{L^2(\Omega_b)} \\
&+ \rho_p (f_5 + f_4, \xi)_{L^2(\omega_p)} + (c_p f_6 - \alpha_p \tilde{\mathcal{K}}(\Delta_{\omega_p} f_4), q_p)_{L^2(\Omega_p)} \\
&+ \int_{\Gamma_I} f_4 q_p(\cdot, -h/2) + \rho_f (\mathbf{f}_7, \boldsymbol{\delta})_{L^2(\Omega_f)}.
\end{aligned}
\end{equation}

The mixed variational formulation (\ref{rel})-(\ref{con}) constitutes the form of the resolvent relation (\ref{resolve}). Accordingly, we will put ourselves in position to invoke use the Babu\u{s}ka-Brezzi Theorem to show that there exists a unique solution to this system of equations (Appendix \ref{sec:babuksa-brezzi}). \\
Note that if the dynamics operator is maximal for some $\lambda > 0$, then it is maximal $\forall \lambda > 0$ \cite{pazy}. For simplicity we have chosen $\lambda = 1$, since $\lambda \neq 1$ only results in different positive coefficients on the data.

\subsection{Coercivity and Continuity}
The bilinear forms $a(\cdot ,\cdot ):\Sigma \times \Sigma $ and $b(\cdot ,\cdot ):\Sigma \times
L^{2}(\Omega _{f})$, as given in (\ref{eq:bilinear_a}) and (\ref{eq:bilinear_b}), are each readily seen to be continuous. (In particular, the Sobolev Trace Theorem and the Inequalitites of Korn and Poincar\'e are used to estimate the boundary integral which partially make up $a(\cdot,\cdot)$.) Moreover, $a(\cdot ,\cdot )$ is $\Sigma$-elliptic:
given $\boldsymbol{\phi} = [\boldsymbol{\eta}, p_b, w, p_p, \mathbf{u}]^T \in \Sigma$, we have that
\begin{align*}
a(\boldsymbol{\phi}, \boldsymbol{\phi}) &= \left( \rho_b \Vert \boldsymbol{\eta} \Vert^2_{L^2(\Omega_b)} + \Vert \boldsymbol{\eta} \Vert_E^2 \right) + \left( c_p \Vert p_b \Vert^2_{L^2(\Omega_b)} + k_b \Vert \nabla p_b \Vert^2_{L^2(\Omega_b)} \right) \\
&+ \left( \rho_p \Vert w \Vert^2_{L^2(\omega_p)} + D \Vert \Delta_{\omega_p} w \Vert^2_{L^2(\omega_p)} + \gamma \Vert w \Vert^2_{L^2(\omega_p)} \right) + \left( c_p \Vert p_p \Vert^2_{L^2(\Omega_p)} + k_p \Vert \partial_s p_p \Vert^2_{L^2(\Omega_p)} \right) \\
&+ \left( \rho_f \Vert \mathbf{u} \Vert^2_{L^2(\Omega_f)} + 2 \mu_f \Vert D(\mathbf{u}) \Vert^2_{L^2(\Omega_f)} + \beta \int_{\Gamma_I} \vert \mathbf{u} \cdot \boldsymbol{\tau} \vert^2 \right) \\
&\geq C(\rho_b, c_p, k_b, \rho_p, D, \gamma, c_p, k_p, \rho_f, \mu_f, \beta) \Vert \boldsymbol{\phi} \Vert^2_{\Sigma}.
\end{align*}

The heart of the matter here is to show the bilinear form $b(\cdot,\cdot): \Sigma \times L^2(\Omega_p)$ fulfills the Babu\v{s}ka-Brezzi condition. That is, there exists a constant $\beta > 0$ such that  
\begin{equation}
\forall \chi \in L^2(\Omega_f) \ \sup_{\boldsymbol{\phi} \in \Sigma \setminus \{ 0 \} } \frac{b(\boldsymbol{\phi}, \chi)}{\Vert \boldsymbol{\phi} \Vert_{\Sigma}} \geq \beta \Vert \chi \Vert_{L^2(\Omega_f)}.
\end{equation}
Our proof follows the same steps as in \cite{AGW}. We have that
\begin{equation}
\sup_{\boldsymbol{\phi} \in \Sigma, \boldsymbol{\phi} \neq 0} \frac{b(\boldsymbol{\phi}, \chi)}{\Vert \boldsymbol{\phi} \Vert_{\Sigma}} \geq \sup_{\substack{\boldsymbol{\phi} \in \Sigma, \mathbf{u} \neq 0 \\ \boldsymbol{\eta} = 0, p_b = 0, \\ p_p = 0, w = 0}} \frac{b(\boldsymbol{\phi}, \chi)}{\Vert \boldsymbol{\phi} \Vert_{\Sigma}} \geq \frac{b(\tilde{\boldsymbol{\phi}}, \chi)}{\Vert \tilde{\boldsymbol{\phi}} \Vert_{\Sigma}} \label{BB1}
\end{equation}
where $\widetilde{\phi }\in \Sigma $ is of the form $\widetilde{\phi }%
=[0,0,0,0,\widetilde{\mathbf{u}}]$, where $\widetilde{\mathbf{u}}\in \mathbf{%
H}_{\#,\ast }^{1}(\Omega _{f})$ is arbitrary. Using the definition of the bilinear form $b(\boldsymbol{\phi}, \pi)$ in (\ref{eq:bilinear_b} and the norm of $\Sigma$, we then have that
\begin{equation}
\frac{b(\tilde{\boldsymbol{\phi}}, \pi)}{\Vert \tilde{\boldsymbol{\phi}} \Vert_{\Sigma}} = -\frac{(\pi, \nabla \cdot \tilde{\mathbf{u}})}{\Vert \mathbf{D}(\tilde{\mathbf{u}}) \Vert_{L^2(\Omega_f)}} \label{BB2}
\end{equation}
With this ratio in mind, we  consider the boundary value problem: \\
Find $\tilde{\mathbf{u}} \in \mathbf{H}_\#^1(\Omega_f)$ such that
\
\begin{eqnarray}
\nabla \cdot \mathbf{\tilde{u}} &=&-\chi \text{ on }\Omega _{f} \nonumber \\
\left. \mathbf{\tilde{u}}\right\vert _{\Gamma _{I}} &=& - \left( \frac{%
\int_{\Omega _{f}}\chi d\Omega _{f}}{meas(\Gamma _{I})}\right) \mu(x) \mathbf{e}%
_{3}\text{ \ on }\Gamma _{I} \nonumber \\
\left. \mathbf{\tilde{u}}\right\vert _{\Gamma _{f}} &=&\mathbf{0}\text{ \ on 
}\Gamma _{f}.  \ref{div}
\end{eqnarray}%
Here, $\mu $ is a cut-off function such that%
\[
\mu (x)=\left\{ 
\begin{array}{l}
1\text{ \ on }\Gamma _{I} \\ 
\text{a nonegative, periodic function, \ on each lateral face }\Gamma \text{ of 
}\Omega _{f},%
\end{array}%
\right. 
\]%
and with supp$(\mu )\cap \Gamma _{f}=\varnothing $. This boundary value problem has a solution \cite{galdi} -- as the compatibility condition for solvability is met --  with the continuous dependence estimate, 
\begin{equation}
\Vert \tilde{\mathbf{u}} \Vert_{H^1(\Omega)} \leq c_1 \Vert \chi \Vert_{L^2(\Omega_f)}. \label{div}
\end{equation}
Applying this solution $\tilde{\mathbf{u}}$ to (\ref{BB2}), we then have 
\begin{equation}
 \frac{b(\tilde{\boldsymbol{\phi}}, \chi)}{\Vert \tilde{\boldsymbol{\phi}} \Vert_{\Sigma}} = -\frac{(\chi, \nabla \cdot \tilde{\mathbf{u}})}{\Vert \mathbf{D}(\tilde{\mathbf{u}}) \Vert_{L^2(\Omega_f)}} = \frac{\Vert \chi \Vert_{L^2(\Omega_f)}^2}{\Vert \mathbf{D}(\tilde{\mathbf{u}}) \Vert_{L^2(\Omega_f)}}.
\end{equation}
From the norm equivalence of $\Vert \mathbf{D}(.) \Vert$ and the $\mathbf{H}^1$ norm, and the continuous dependence estimate (\ref{div}), we then have that
$$\frac{\Vert \chi \Vert_{L^2(\Omega_f)}^2}{\Vert \mathbf{D}(\tilde{\mathbf{u}}) \Vert_{L^2(\Omega_f)}} \geq c_2 \frac{\Vert \chi \Vert^2_{L^2(\Omega_f)}}{\Vert \chi \Vert_{L^2(\Omega_f)}} = c_2 \Vert \chi \Vert_{L^2(\Omega_f)}.$$
This estimate and (\ref{BB1}) yield now, for given $\chi \in L^2(\Omega_f)$
$$\sup_{\boldsymbol{\phi} \in \Sigma, \boldsymbol{\phi} \neq 0} \frac{b(\boldsymbol{\phi}, \chi)}{\Vert \boldsymbol{\phi} \Vert_{\Sigma}} \geq c_2 \Vert \chi \Vert_{L^2(\Omega_f)}.$$
for arbitrary $\pi \in L^2(\Omega_f)$. 
Given this inf-sup estimate, relative to continuous $b(\cdot,\cdot)$, and as $a(\cdot.\cdot)$ is continuous and $\Sigma$-elliptic, the Babu\u{s}ka-Brezzi Theorem provides for a pair $\phi \in \Sigma$ and $\pi \in L^2$ which uniquely solve the system of equations (\ref{rel})-(\ref{con}).

\subsection{Recovery of \texorpdfstring{$\mathcal D(\mathbf A)$}{}}

We wish now to show that our unique solution to the Babu\v{s}ka-Brezzi system (\ref{rel})-(\ref{con}), $(\boldsymbol{\phi}, \pi) \in \Sigma \times L^2(\Omega_f)$, gives rise to an element in $\mathcal{D}(\mathbf{A})$. (Again, $\boldsymbol{\phi} = [\boldsymbol{\eta}, p_b, w, p_p, \mathbf{u}]^T $.) 

\medskip

\textbf{{\large Part 1.}} We start by extracting some preliminary relations from the variational solution to (\ref{rel})-(\ref{con}). Let test function $\psi = [\mathbf{k}, q_b, \xi, q_p, \boldsymbol{\delta}]^T \in \Sigma$, where $\mathbf{k} \in [\mathcal{D}(\Omega _{b})]^{3}$, $q_b =0, q_p  =0 ,\xi = 0$, and $\boldsymbol{\delta} = \mathbf{0}$. For this particular test function $\boldsymbol{\psi}$ we then have from the variational relation (\ref{rel}) that
$$\rho_b (\boldsymbol{\eta}, \mathbf{k})_{L^2(\Omega_b)} + a_E(\boldsymbol{\eta}, \mathbf{k}) - \alpha_b (p_b, \nabla \cdot \mathbf{k})_{L^2(\Omega_b)} = \rho_b (\mathbf{f}_2 + \mathbf{f}_1, \mathbf{k}).$$
Given the regularity in $\Sigma$, we have from (\ref{rel}) that
\begin{equation}
    \rho_b \boldsymbol{\eta} - \nabla \cdot \sigma^E(\boldsymbol{\eta}) + \alpha_b \nabla p_b = \rho_b (\mathbf{f}_2 + \mathbf{f}_1),  \label{uno}
    \end{equation}
in the sense of distributions, and by the density of $\mathcal{D}(\Omega_{b})$, in $L^2$-sense. Thus,
\begin{equation}
\nabla \cdot \sigma^E(\boldsymbol{\eta}) =\rho_b \boldsymbol{\eta}   + \alpha_b \nabla p_b - \rho_b (\mathbf{f}_2 + \mathbf{f}_1)  \in \mathbf{L}^2(\Omega_b).  \label{dos}
\end{equation}
This regularity subsequently yields via an energy method that 
\begin{equation}
    \sigma^E(\boldsymbol{\eta})\cdot \mathbf{n} \in \mathbf{H}^{-1/2}(\partial \Omega_b). \label{tres}
\end{equation}
This same train of logic can be invoked for each of the solution variables, by varying the choices of test function $\boldsymbol{\psi}$, so as to derive the following relations:
\begin{align}
k_b \Delta p_b  = & c_b p_b   + \alpha_b \nabla \cdot \boldsymbol{\eta} - c_b f_3 - \alpha_b \nabla \cdot \mathbf{f}_1  \in  L^2(\Omega_b)  \label{l.1} \\
k_b \nabla p_b \cdot \mathbf{n_b} & \in H^{-1/2}(\partial \Omega_b) \label{l.2} \\
~~  \nonumber\\
 - \mu_f \Delta \mathbf{u} + \nabla \pi &= - \rho_f \mathbf{u} + \rho_f \mathbf{f}_7  ~~ \in  \mathbf{L}^2(\Omega_f)  \label{l.3} \\
 \nabla \cdot \mathbf{u} & =0 \text{~ in ~} \Omega_f \text{~(from ~ (\ref{con}))} \label{3.5} \\ 
 \sigma_f \cdot \boldsymbol{e}_3 & \in \mathbf{H}^{-1/2}(\Omega_f) \label{l.4} \\
 ~~  \\
 c_p p_p -k_p\partial_s^2p_p -\alpha_p \tilde{\mathcal{K}}(\Delta w) & = c_p f_6 - \alpha_p \tilde{\mathcal{K}}(\Delta_{\omega_p} f_4)  \in  L^2(\Omega_p), \label{l.6}
\end{align}

The situation with the \textquotedblleft thin\textquotedblright\ plate
component of the bilinear form is a bit more subtle, because of the matching
displacements constraint in $\Sigma $. Taking test function $\boldsymbol{\psi} $ in (\ref%
{rel}) as $\boldsymbol{\psi} =[\mathbf{k},0,\xi ,0,0]^{T}\in \Sigma $, then we have  
\begin{eqnarray*}
&&\rho _{b}\left( \mathbf{\eta },\mathbf{k}\right) _{L^{2}(\Omega
_{b})}+a_{E}\left( \mathbf{\eta },\mathbf{k}\right) -\alpha _{b}\left(
p_{b},\nabla \cdot \mathbf{k}\right) _{L^{2}(\omega _{b})}+\rho _{p}\left(
w,\xi \right) _{L^{2}(\omega _{p})}+D\left( \Delta _{\omega _{p}}w,\Delta
_{\omega _{p}}\xi \right) _{L^{2}(\omega _{p})} \\
&&+\gamma \left( w,\xi \right) _{L^{2}(\omega _{p})}+\alpha _{p}\left(
\Delta _{\omega _{p}}\mathcal{K}\left( p_{p}\right) ,\xi \right)
_{_{L^{2}(\omega _{p})}}-\int_{\Gamma _{I}}p_{p}(\cdot ,-h/2)\xi d\Gamma _{I}
\\
&=&\rho _{b}\left( \mathbf{f}_{2}+\mathbf{f}_{1},\mathbf{k}\right)
_{L^{2}(\Omega _{b})}+\rho _{p}\left( f_{5}+f_{4},\xi \right) _{L^{2}(\omega
_{p})}.
\end{eqnarray*}%
Applying (\ref{uno}) and (\ref{dos}) to this expression, and using the fact
that $\left. \mathbf{k}\right\vert _{\Gamma _{I}}=\xi $, we then have the
relation 
\[
\begin{array}{l}
\rho _{p}\left( w,\xi \right) _{L^{2}(\omega _{p})}+D\left( \Delta w,\Delta
\xi \right) _{L^{2}(\omega _{p})}+\gamma \left( w,\xi \right) _{L^{2}(\omega
_{p})}-\left\langle \sigma _{b}(\mathbf{\eta },p_b )\mathbf{e}_{3}\cdot 
\mathbf{e}_{3},\xi \right\rangle _{\Gamma _{I}} \\ 
+\alpha _{p}\left( \Delta _{\omega _{p}}\mathcal{K}\left( p_{p}\right) ,\xi
\right) _{_{L^{2}(\omega _{p})}}-\int_{\Gamma _{I}}p_{p}(\cdot ,-h/2)\xi
d\Gamma _{I} \\ 
\text{ \ \ }=\rho _{p}\left( f_{5}+f_{4},\xi \right) _{L^{2}(\omega _{p})}.%
\end{array}%
\]%
Since, for any function $h\in H_{0}^{2}(\Gamma _{I})$, one can find $\varpi
\in H_{\#,\ast }^{1}(\Omega _{b})\cap H^{2}(\Omega _{b})$ -- see the
characterization of \cite{B-G}, below in (\ref{range})  -- such that $\left.
\varpi \right\vert _{\Gamma _{I}}=h$, we then have 
\begin{equation}
\rho _{p}w+D\Delta ^{2}w+\gamma w+\alpha _{p}\Delta _{\omega _{p}}\mathcal{K}%
\left( p_{p}\right) -\sigma _{b}(\mathbf{\eta },p_b )\mathbf{e}_{3}\cdot 
\mathbf{e}_{3}-p_{p}(\cdot ,-h/2)=\rho _{p}[f_{5}+f_{4}]\text{ \ in }%
L^{2}(\Gamma _{I}).  \label{l.5}
\end{equation}

\medskip
\textbf{{\large Part 2.}}
Here, we establish \textquotedblleft hidden regularity'' for the solution pair $[\boldsymbol{\phi},\pi] \in \Sigma \times L^2(\Omega_f)$ for (\ref{rel})-(\ref{con}), where again, $ \boldsymbol{\phi} =[\boldsymbol{\eta}, p_b, w, q_b, \mathbf{u}]^T$. \newline 
(i) To start,
with%
\begin{equation}
\partial \Omega _{f}=\bigcup\limits_{j=1}^{6}\Gamma _{j}, \label{sides}
\end{equation}
where each $\Gamma _{j}$ denoting one of the six sides of $\Omega _{f}$ then%
\begin{equation}
\left\{ \left. \pi \right\vert _{\Gamma _{j}},\left. \frac{\partial \pi }{%
\partial \mathbf{n}_{f}}\right\vert _{\Gamma j}\right\} \in H^{-\frac{1}{2}%
}(\Gamma _{j})\times H^{-\frac{3}{2}}(\Gamma _{j})\text{, for }j=1,...,6.
\label{T.1}
\end{equation}%
To show this containment, we will use the characterization of the (global)
Sobolev Trace Map, $\gamma =\left( \gamma _{0},\gamma _{1}\right)
:H^{2}(\Omega _{f})\rightarrow H^{1}(\partial \Omega _{f})\times
L^{2}(\partial \Omega _{f})$ (on Lipschitz domain $\Omega _{f}$). That is,%
\[
f\rightarrow \gamma (f)=\left[ \gamma _{0}(f)=\left. f\right\vert _{\partial
\Omega _{f}},\gamma _{1}(f)=\left. \frac{\partial f}{\partial \mathbf{n}_{f}}%
\right\vert _{\partial \Omega _{f}}\right] .
\]%
For this mapping, one has%
\begin{equation}
Range(\gamma )=\left\{ [h_{0},h_{1}]\in H^{1}(\partial \Omega _{f})\times
L^{2}(\partial \Omega _{f}):\nabla _{\partial \Omega _{f}}h_{0}+h_{1}\mathbf{%
n}_{f}\in \mathbf{H}^{\frac{1}{2}}(\partial \Omega _{f})\right\} 
\label{range}
\end{equation}%
(see Theorem 5, p. 702, of \cite{B-G}. Here, $\nabla _{\partial \Omega _{f}}$
is the tangential gradient.)

With this result in hand and $\Gamma \equiv
\Gamma _{j}$, for some index $j$, let 
\begin{equation}
g_{0}\in H_{0}^{\frac{3}{2}+\epsilon }(\Gamma ),g_{1}\in H_{0}^{\frac{1}{2}%
+\epsilon }(\Gamma )  \label{zero}
\end{equation}%
be given; subsequently, let $\widetilde{g_{0}}$ and $\widetilde{g_{1}}$
denote their respective \emph{extensions by zero} onto all of $\partial
\Omega _{f}$. Then by Theorem 3.33, p. 95, of \cite{Mc}, 
\begin{equation}
\widetilde{g_{0}}\in H^{\frac{3}{2}+\epsilon }(\partial \Omega _{f}),%
\widetilde{g_{1}}\in H^{\frac{1}{2}+\epsilon }(\partial \Omega _{f}).
\label{ex}
\end{equation}%
As such, from (\ref{range}), we see that $\left[ \widetilde{g_{0}},%
\widetilde{g_{1}}\right] \in Range(\gamma )$. Consequently, if we invoke the
continuous right inverse $\gamma ^{+}$ of $\gamma $ -- i.e., $\gamma ^{+}\in 
\mathcal{L(}Range(\gamma ),H^{2}(\Omega _{f})$ satisfies $\gamma \gamma
^{+}=I$, then 
\begin{equation}
\left\langle \frac{\partial \pi }{\partial \mathbf{n}_{f}}%
,g_{0}\right\rangle _{\Gamma }-\left\langle \pi ,g_{1}\right\rangle _{\Gamma
}=\left( \pi ,-\Delta \gamma ^{+}(\left[ \widetilde{g_{0}},\widetilde{g_{1}}%
\right] \right) _{\Omega _{f}}+0,  \label{z1}
\end{equation}

after using the fact that $\pi $ is harmonic. (Essentially, we are mimicking
the proof of the analogous Proposition , p.265 of \cite{A-Dv}, which was
undertaken in the case of a smooth fluid geometry.) \ Estimating RHS, in
part using the fact that $\pi $ is the pressure solution component of (\ref%
{rel})-(\ref{con}), we have 
\begin{eqnarray*}
\left\vert \left( \pi ,-\Delta \gamma ^{+}(\left[ \widetilde{g_{0}},%
\widetilde{g_{1}}\right] \right) _{\Omega _{f}}\right\vert  &\leq
&C\left\Vert \pi \right\Vert _{L^{2}(\Omega )}\left\Vert \Delta \gamma ^{+}(%
\left[ \widetilde{g_{0}},\widetilde{g_{1}}\right] \right\Vert _{L^{2}(\Omega
_{f})} \\
&\leq &C\left\Vert \mathbf{F}\right\Vert _{\mathbf{X}}\left\Vert \nabla
_{\partial \Omega _{f}}\widetilde{g_{0}}+\widetilde{g_{1}}\mathbf{n}%
_{f}\right\Vert _{\mathbf{H}^{\frac{1}{2}}(\Omega _{f})},
\end{eqnarray*}%
after using characterization (\ref{range}). Since the zero extensions $%
\widetilde{g_{i}}$ come from the $g_{i}$ in (\ref{zero}), we then have%
\begin{equation}
\left\vert \left( \pi ,-\Delta \gamma ^{+}(\left[ \widetilde{g_{0}},%
\widetilde{g_{1}}\right] \right) _{\Omega _{f}}\right\vert \leq C\left\Vert 
\mathbf{F}\right\Vert _{\mathbf{X}}\left( \left\Vert g_{0}\right\Vert _{H^{%
\frac{3}{2}}(\Gamma )}+\left\Vert g_{1}\right\Vert _{H^{\frac{1}{2}}(\Gamma
)}\right) .  \label{z2}
\end{equation}%

The combination of (\ref{z1}) and (\ref{z2}), and the density of $%
H_{0}^{s+\epsilon }(\Gamma )$ in $H_{0}^{s}\,$, gives now%
\[
\left\vert \left\langle \frac{\partial \pi }{\partial \mathbf{n}_{f}}%
,g_{0}\right\rangle _{\Gamma }-\left\langle \pi ,g_{1}\right\rangle _{\Gamma
}\right\vert \leq C\left\Vert \mathbf{F}\right\Vert _{\mathbf{X}}\left(
\left\Vert g_{0}\right\Vert _{H^{\frac{3}{2}}(\Gamma )}+\left\Vert
g_{1}\right\Vert _{H^{\frac{1}{2}}(\Gamma )}\right) \text{, for all }%
g_{0}\in H_{0}^{\frac{3}{2}}(\Gamma )\text{ and }g_{1}\in H_{0}^{\frac{1}{2}%
}(\Gamma )\text{.}
\]%
This estimate now yields (\ref{T.1}). 
\newline
\medskip 
(ii) In turn, with given cube side $\Gamma $, as in (\ref{sides}), let $%
\mathbf{g}\in \mathbf{H}_{0}^{\frac{1}{2}+\epsilon }(\Gamma )$, and so its
extension by zero $\widetilde{\mathbf{g}}\in \mathbf{H}^{\frac{1}{2}%
+\epsilon }(\partial \Omega _{f})$. Therewith, and accompanied by the right
continuous inverse $\gamma _{0}^{+}\in \mathcal{L}(\mathbf{H}^{\frac{1}{2}%
}(\partial \Omega _{f}),\mathbf{H}^{1}(\Omega _{f}))$ of the zero order
Sobolev Trace map $\gamma _{0}$, we have 
\begin{eqnarray}
\left\langle 2\mu _{f}\mathbf{D}(\mathbf{u})\cdot \mathbf{n}_{f},\mathbf{g}%
\right\rangle _{\Gamma } &=&2\mu _{f}\left( \mathbf{D}(\mathbf{u}),\mathbf{D}%
(\gamma _{0}^{+}[\widetilde{\mathbf{g}}])\right) _{L^{2}(\Omega _{f})}+\mu
_{f}\left( \Delta \mathbf{u},\gamma _{0}^{+}[\widetilde{\mathbf{g}}]\right)
_{L^{2}(\Omega _{f})}  \nonumber \\
&=&\left( \nabla \pi ,\gamma _{0}^{+}[\widetilde{\mathbf{g}}]\right)
_{L^{2}(\Omega _{f})}+\rho _{f}\left( \mathbf{u},\gamma _{0}^{+}[\widetilde{%
\mathbf{g}}]\right) _{L^{2}(\Omega _{f})}-\rho _{f}\left( \mathbf{f}%
_{7},\gamma _{0}^{+}[\widetilde{\mathbf{g}}]\right) _{L^{2}(\Omega _{f})}.
\label{z3}
\end{eqnarray}%
after using the Stokes flow (\ref{l.3}). With respect to the first term on
RHS,%
\begin{eqnarray*}
\left\vert \left( \nabla \pi ,\gamma _{0}^{+}[\widetilde{\mathbf{g}}]\right)
_{L^{2}(\Omega _{f})}\right\vert  &\leq &\left\vert \left\langle \pi ,%
\mathbf{g}\right\rangle _{\Gamma }\right\vert +\left\vert \left( \pi ,\nabla
\cdot \gamma _{0}^{+}[\widetilde{\mathbf{g}}]\right) _{L^{2}(\Omega
_{f})}\right\vert  \\
&\leq &C\left\Vert \mathbf{F}\right\Vert _{\mathbf{X}}\left\Vert \mathbf{g}%
\right\Vert _{\mathbf{H}^{\frac{1}{2}}(\Gamma )},
\end{eqnarray*}%
after using (\ref{T.1}) and the fact that $\pi $ is the pressure solution
component of (\ref{rel})-(\ref{con}), Applying this estimate to (\ref{z3}),
and estimating -- using the fact that $\mathbf{u}$ is the solution component
of (\ref{rel})-(\ref{con}) -- we have 
\[
\left\vert \left\langle 2\mu _{f}\mathbf{D}(\mathbf{u})\cdot \mathbf{n}_{f},%
\mathbf{g}\right\rangle _{\Gamma }\right\vert \leq C\left\Vert \mathbf{F}%
\right\Vert _{\mathbf{X}}\left\Vert \mathbf{g}\right\Vert _{\mathbf{H}^{%
\frac{1}{2}}(\Gamma )}\text{, for given }\mathbf{g}\in \mathbf{H}_{0}^{\frac{%
1}{2}+\epsilon }(\Gamma ).
\]%
The density of  $\mathbf{H}_{0}^{\frac{1}{2}+\epsilon }(\Gamma )$ in $%
\mathbf{H}^{\frac{1}{2}}(\Gamma )$ gives now,%
\begin{equation}
\left\Vert 2\mu _{f}\mathbf{D}(\mathbf{u})\cdot \mathbf{n}_{f}\right\Vert _{%
\mathbf{H}^{-\frac{1}{2}}(\Gamma )}\leq C\left\Vert \mathbf{F}\right\Vert _{%
\mathbf{X}}\text{, \ for given side }\Gamma \text{ of cube }\Omega _{f}.
\label{T.2}
\end{equation}

(iii) Lastly,  for a given face $\Gamma $ of $\Omega _{f}$, take $\pm \mathbf{%
e}_{j}$ to be its normal vector. Then dotting both sides of the equation (%
\ref{l.3}),  with respect to $\mathbf{e}_{j}$ and resstricting its trace to $%
\Gamma $, we then have 
\[
\mu _{f}\Delta \mathbf{u}\cdot \mathbf{e}_{j}=\nabla \pi \cdot \mathbf{e}%
_{j}+\rho _{f}\mathbf{u}\cdot \mathbf{e}_{j}-\rho _{f}\mathbf{f}_{7}\cdot 
\mathbf{e}_{j}\text{ in }\Omega _{f}.
\]%
Subsequently, taking the Dirichlet trace of both sides, restricting to $%
\Gamma $ and then estimating, we then have 
\[
\left\Vert \mu _{f}\Delta \mathbf{u}\cdot \mathbf{e}_{j}\right\Vert _{H^{-%
\frac{3}{2}}(\Gamma )}\leq \left\Vert \nabla \pi \cdot \mathbf{e}%
_{j}\right\Vert _{H^{-\frac{3}{2}}(\Gamma )}+C\left\Vert \mathbf{F}%
\right\Vert _{\mathbf{X}}.
\]%
Invoking the estimate (\ref{T.1}) now provides the trace regularity,%
\begin{equation}
\left\Vert \mu _{f}\Delta \mathbf{u}\cdot \mathbf{n}_{f}\right\Vert _{H^{-%
\frac{3}{2}}(\Gamma )}\leq C\left\Vert \mathbf{F}\right\Vert _{\mathbf{X}},%
\text{ \ for any given side of the cube }\Omega _{f}\text{.}  \label{T.3}
\end{equation}
\medskip 

We also note that since solution component $p_{p}\in H^{0,0,2}(\Omega _{p})$%
, we have from (\ref{trace}) and (\ref{trace2}) the traces 
\begin{equation}
\left\{ \left. p_{p}\right\vert _{\omega _{p}^{+}},\left. \partial
_{s}p_{p}\right\vert _{\omega _{p}^{+}},\left. p_{p}\right\vert _{\omega
_{p}^{-}},\left. \partial _{s}p_{p}\right\vert _{\omega _{p}^{-}}\right\}
\in \lbrack H^{-\frac{1}{2}}(\omega _{p}^{+})]^{2}\times \lbrack H^{-\frac{1%
}{2}}(\omega _{p}^{-})]^{2}.  \label{T.4}
\end{equation}
\medskip
\medskip

\medskip
\textbf{{\large Part 3.}}
(a) We establish that the trace term $\left[ \sigma _{f}(\mathbf{u},\pi )%
\mathbf{n}_{f}\right] _{\partial \Omega _{f}}$ is $\mathbf{H}_{\#}^{-\frac{1%
}{2}}(\partial \Omega _{f})$. Given two \emph{lateral} sides $\Gamma _{a}$ $%
\ $and $\Gamma _{b}$, say, of the cube $\Omega _{f}$, let $\mathbf{g}$ $\in 
\mathbf{H}_{0}^{\frac{1}{2}+\epsilon }(\Gamma _{a})$ be given. Then one can
find $\mathbf{v}\in \mathbf{H}^{1}(\Omega _{f})$ which satisfies the BVP%
\begin{eqnarray}
\nabla \cdot \mathbf{v} &=0&\text{ in }\Omega _{f}  \nonumber \\
\mathbf{v} &\mathbf{=g}&\text{ on }\Gamma _{a}\cup \Gamma _{b}  \nonumber \\
\mathbf{v} &\mathbf{=0}&\text{ on }\partial \Omega _{f}\diagdown (\Gamma
_{a}\cup \Gamma _{b}).\text{ }  \label{3.a}
\end{eqnarray}%
Note the underlying compatibility condition for solvability is met, since%
\[
\int_{\Gamma _{a}}\mathbf{g}\cdot \mathbf{n}_{f}d\Gamma _{a}+\int_{\Gamma
_{b}}\mathbf{g}\cdot \mathbf{n}_{f}d\Gamma _{b}=0\text{.}
\]%
With this variable, and unit normal vector $\pm \mathbf{e}_{j}$ of $\Gamma
_{a}$, $j=1$ or 2, we then have%
\begin{eqnarray*}
&&\left\langle 2\mu _{f}\mathbf{D}(\mathbf{u})(\pm \mathbf{e}_{j})-\pi (\pm 
\mathbf{e}_{j}),\mathbf{g}\right\rangle _{\Gamma _{a}}+2\mu _{f}\left\langle 
\mathbf{D}(\mathbf{u})(\mp \mathbf{e}_{3})-\pi (\mp \mathbf{e}_{j}),\mathbf{g%
}\right\rangle _{\Gamma _{a}} \\
&&2\mu _{f}\left( \mathbf{D}(\mathbf{u}),\mathbf{D}(\mathbf{v})\right)
_{\Omega _{f}}+\mu _{f}\left( \Delta \mathbf{u},\mathbf{v}\right) _{\Omega
_{f}}-\left( \nabla \pi ,\mathbf{v}\right) _{\Omega _{f}} \\
&=&\left( \pi ,\nabla \cdot \mathbf{v}\right) _{\Omega _{f}} \\
&&\text{(after using (\ref{rel}) and (\ref{l.3})} \\
&=&0\text{.}
\end{eqnarray*}%
Thus,%
\begin{equation}
\left[ 2\mu _{f}\mathbf{D}(\mathbf{u})\mathbf{n}_{f}-\pi \mathbf{n}_{f}%
\right] _{\partial \Omega _{f}}\in \mathbf{H}_{\#}^{-\frac{1}{2}}(\partial
\Omega _{f}).  \label{3.b}
\end{equation}

(b) Invoking verbatim the computations in \cite[(4.26), pp. 19--20]{AGW} we have that
\begin{equation}
 2\mu _{f}\mathbf{D}(\mathbf{u})\mathbf{n}_{f} \in \mathbf{H}_{\#}^{-\frac{1}{2}}(\partial
\Omega _{f});  \label{3.c}
\end{equation}
this and (\ref{3.b}) give in turn, 
\begin{equation}
  \pi \in \mathbf{H}_{\#}^{-\frac{1}{2}}(\partial
\Omega _{f});  \label{3.c2}.
\end{equation}
(c) Proceeding in an analogous manner, we have with respect to the Biot PDE component of solution $\boldsymbol{\phi}$ of (\ref{rel})-(\ref{con}),
\begin{equation}
\left. \nabla p_{b}\cdot \mathbf{n}_{f}\right\vert _{\partial \Omega
_{b}}\in H_{\#}^{-\frac{1}{2}}(\partial \Omega _{b});\text{\ }\left. \sigma
_{b}(\mathbf{\eta })\right\vert _{\partial \Omega _{b}}\sigma _{b}(\mathbf{%
\eta })\in H_{\#}^{-\frac{1}{2}}(\partial \Omega _{b}).  \label{3.d}
\end{equation}

\textbf{{\large Part 4.}}
We conclude the domain recovery by showing that the various boundary interface conditions are met. With respect to our variational solution $(\phi, \pi) \in \Sigma \times L^2(\Omega_f)$ of (\ref{rel})-(\ref{con}) and given $\mathbf{\psi }=[\mathbf{k},q_{b},\xi ,q_{p},\boldsymbol{\delta }]\in \Sigma $, we have the following identities, upon applying Green's Theorems to \ref{dos})-(\ref{l.6}) and subsequently invoking the boundary traces now available in (\ref{T.4}), (\ref{3.c}),(\ref{3.c2}) and (\ref{3.d}):
\begin{align}
\left( \rho_f \mathbf{f}_7,\boldsymbol{\delta}\right)_{L^{2}(\Omega _{f})} & = \rho_f\left( \boldsymbol{u} ,\boldsymbol{\delta}\right)_{L^{2}(\Omega _{f})} + (- \nu \Delta \mathbf{u} + \nabla \pi, \boldsymbol{\delta})_{L^2(\Omega_f)} \nonumber \\
          & = \rho_f\left(  \boldsymbol{u} ,\boldsymbol{\delta}\right)_{L^{2}(\Omega _{f})}+ 2\mu_f (\mathbf{D}(\mathbf{u}), \mathbf{D}(\boldsymbol{\delta}))_{L^2(\Omega_f)} - (\pi, \nabla \cdot \boldsymbol{\delta})_{L^2(\Omega_f)} - \langle \sigma_f(\mathbf{u}, \pi) \mathbf{e}_3, \boldsymbol{\delta} \rangle_{\Gamma_I}; \label{fluid}
\end{align}
\begin{align}
     \left( \rho _{b}[\mathbf{f}_{2}+\mathbf{f}_{1}],%
\mathbf{k}\right) _{\mathbf{L}^{2}(\Omega _{b})} &= \rho_{b} \left( \boldsymbol{\eta },%
\mathbf{k}\right) _{\mathbf{L}^{2}(\Omega _{b})} - (\nabla \cdot \sigma_E(\boldsymbol{\eta}) - \alpha_b \nabla p_b, \mathbf{k})_{L^2(\Omega_b)} \nonumber \\
    & = \rho_{b} \left( \boldsymbol{\eta },%
\mathbf{k}\right) _{\mathbf{L}^{2}(\Omega _{b})} + a_E(\boldsymbol{\eta}, \mathbf{k}) - \alpha_b(p_b, \nabla \cdot \boldsymbol{k})_{L^2(\Omega_b)} + \langle \sigma_b(\boldsymbol{\eta}, p_b) \mathbf{e}_3, \mathbf{k} \rangle_{\Gamma_I}; 
\end{align}

\begin{eqnarray}
  \left( c_{b}f_{3}+\alpha
_{b}\nabla \cdot \mathbf{f}_{1},q_{b}\right) _{L^{2}(\Omega _{b})} 
= \left( c_{b}p_{b}+\alpha _{b}\nabla \cdot \boldsymbol{\eta },q_{b}\right) _{L^{2}(\Omega _{b})} -\left(
k_{b}\Delta p_{b},q_{b}\right) _{L^{2}(\Omega _{b})}  \nonumber \\
= \left( c_{b}p_{b}+ \alpha _{b}\nabla \cdot \boldsymbol{\eta },q_{b}\right) _{L^{2}(\Omega _{b})} +\left( k_{b}\nabla p_{b},\nabla q_{b}\right) _{L^{2}(\Omega
_{b})}  +\left\langle k_{b}\partial _{\mathbf{e}_{3}}p_{b},q_{b}\right\rangle
_{\Gamma _{I}};  \label{pb}
\end{eqnarray}

\begin{eqnarray}
&&\left( \rho _{p}[f_{5}+f_{4}],\xi \right) _{L^{2}(\omega _{p})}  \nonumber
\\
&=&\rho _{p}\left( w,\xi \right) _{L^{2}(\omega _{p})}+D\left( \Delta_{\omega_p}
w,\Delta_{\omega_p} \xi \right) _{L^{2}(\omega _{p})}+\gamma \left( w,\xi \right)
_{L^{2}(\omega _{p})}+\left( \alpha _{p}\Delta_{\omega_p} \mathcal{K}(p_{p}),\xi
\right) _{L^{2}(\omega _{p})}  \nonumber \\
&&-\left\langle \sigma _{b}(\mathbf{\eta },p_{b})\mathbf{e}_{3}\cdot \mathbf{%
e}_{3},\xi \right\rangle _{\Gamma _{I}}-\left( p_{p}(\cdot ,-h/2),\xi
\right) _{L^{2}(\omega _{p})}  \label{w}
\end{eqnarray}

\begin{equation}
\begin{array}{l}
\left( c_{p}f_{6}-\alpha _{p}\mathcal{\tilde{K}}(\Delta _{\omega
_{p}}f_{4}),q_{p}\right) _{L^{2}(\Omega _{p})} \\ 
\text{ \ \ \ \ }=c_{p}\left( p_{p},q_{p}\right) _{L^{2}(\Omega _{p})}-\alpha
_{p}\left( \mathcal{\tilde{K}}(\Delta _{\omega _{p}}w),q_{p}\right)
_{L^{2}\Omega _{p})}-\left( k_{p}\partial _{s}^{2}p_{p},q_{p}\right)
_{L^{2}(\Omega _{p})} \\ 
=c_{p}\left( p_{p},q_{p}\right) _{L^{2}(\Omega _{p})}-\alpha _{p}\left( 
\mathcal{\tilde{K}}(\Delta _{\omega _{p}}w),q_{p}\right) _{L^{2}\Omega
_{p})}+\left( k_{p}\partial _{s}p_{p},\partial _{s}q_{p}\right)
_{L^{2}(\Omega _{p})} \\ 
\text{ \ \ \ \ }-\left\langle k_{p}\partial _{s}p_{p},q_{p}\right\rangle
_{\omega _{p}^{+}}+\left\langle k_{p}\partial _{s}p_{p},q_{p}\right\rangle
_{\omega _{p}^{-}}.%
\end{array}
\label{pp2}
\end{equation}

 Now, we add the relations (\ref{fluid})-(\ref{pp2}), and then subtract from the resulting sum the variational relation (\ref{rel}). In this way, we obtain the following relation, for all $\mathbf{\psi }=[\mathbf{k},q_{b},\xi ,q_{p},\boldsymbol{\delta }]\in \Sigma $:

\begin{multline}
     - \langle \sigma_f(\bm{u}, \pi) \bm{e}_3, \bm{\delta} \rangle_{\Gamma_I} + \langle k_b \partial_{\bm{e}_3} p_b, q_b \rangle_{\Gamma_I} - \langle k_p \partial_s p_p, q_p \rangle_{\omega_p^+} + \langle k_p \partial_s p_p, q_p \rangle_{\omega_p^-} - \beta \int_{\Gamma_I} (\bm{u} \cdot \bm{\tau}) (\bm{\delta} \cdot \bm{\tau}) d\Gamma_I \\ - \int_{\Gamma_I} [w - (\bm{u} \cdot \bm{e}_3)] q_p(\cdot, -h/2) d \Gamma_I - \int_{\Gamma_I} p_p(\cdot, -h/2) (\bm{\delta} \cdot \bm{e}_3) d\Gamma_I = - \int_{\Gamma_I} f_4 q_p(\cdot, -h/2) d\Gamma_I.
     \label{bnd}
\end{multline}
(a) At this point, we set
\begin{equation}
  v = w - f_4 \in H_0(\omega_p).
\end{equation}
Now, given $h \in H_0^{\frac{1}{2} + \varepsilon}(\Gamma_I)$, then let $\tilde{h}$ denote its extension by zero. So, $\tilde{h} \in H_0^{\frac{1}{2} + \varepsilon}(\partial \Omega_p)$. Then by the Sobolev Trace Theorem, $\exists \Theta(\tilde{h}) \in H'(\Omega_p) \ni \Theta(\tilde{h})\vert_{\omega_p^-} = h$ and $\Theta(\tilde{h})\vert_{\omega_p^+} = 0.$ Taking the $\Sigma$- test function $\psi = [\bm{0},0,0,\Theta(h),\bm{0}]^T$ in \ref{bnd}, we have
\begin{equation}
    \langle k_p \partial_s p_p - [v - (\bm{u} \cdot \bm{e}_3], h \rangle_{\Gamma_I} = 0, \forall h \in H_0^{\frac{1}{2} + \varepsilon}(\Gamma_I)
\end{equation}
and in fact $\forall h \in H^{\frac{1}{2}}(\Gamma_I)$, by density. Thus
\begin{equation}
    k_p \partial_s p_p \vert_{\omega_p^-} = v - (\bm{u} \cdot \bm{e}_3).
\end{equation}
(b) Given $g \in H_0^{\frac{1}{2} + \varepsilon}(\Gamma_I)$, then its zero extension $\tilde{g}$ onto $\Omega_f$ is in $H^{\frac{1}{2} + \varepsilon}(\partial \Omega_f).$ By the Sobolev Trace Theorem there then exists $\phi \in H^1_{\#,*}(\Omega_f) \ni \phi \vert_{\Gamma_I} = \tilde{g} \vert_{\Gamma_I} = g.$ Taking then $\psi = [\bm{0},0,0,0,\phi\bm{e}_3]^T,$ \ref{bnd} becomes
\begin{equation*}
    - \langle \sigma_f \bm{e}_3 \cdot \bm{e}_3, g \rangle_{\Gamma_I} - \int_{\Gamma_I} p_p(\cdot, -h/2) g d\Gamma_I = 0,
\end{equation*}
for given $g \in H_0^{\frac{1}{2}+\varepsilon}(\Gamma_I)$. So,
\begin{equation}
- \sigma_f \bm{e}_3 \cdot \bm{e}_3 \vert_{\Gamma_I} = p_p \vert_{\omega_p^-}.
    \label{normal}
\end{equation}
(c) with the same function $\phi \in H^1_{\#,*}(\Omega_f)$ as in (b), taking in \ref{bnd} with test function $\psi = [\bm{0},0,0,0,\phi\bm{e}_j]^T, j = 1,2,$ we have
\begin{equation}
    - \sigma_f \bm{e}_3 \cdot \bm{\tau} \vert_{\Gamma_I} = \beta (\bm{u} \cdot \bm{\tau}) \vert_{\Gamma_I}.
\end{equation}
(d) Lastly, using again the zero extension $\tilde{g} \in H^{\frac{1}{2} + \varepsilon}(\partial \Omega_f)$ of a $H_0^{\frac{1}{2} + \varepsilon}(\Gamma_I)$ function $g$, the Sobolev Trace Theorem insures $\exists \Phi_b \in H^1_{\#,*}(\Omega_b) \ni \Phi_b \vert_{\Gamma_I} = g$. Likewise this Theorem provides for $\Phi_p \in H^1(\Omega_p) \ni$
\begin{equation*}
\Phi_p = \begin{cases}
    g & \text{on} \ \omega_p^+ \\
    0 & \text{on} \ \partial \Omega_p \setminus \omega_p^+.
\end{cases}
\end{equation*}
Subsequently taking $\Sigma$ - test function $\psi = [\bm{0}, \Phi_b, 0, \Phi_p, \bm{0}]^T$ in \ref{bnd}, we have
\begin{equation*}
    \langle k_b \partial_{\bm{e}_3} p_b - k_p \partial_s p_p, g \rangle_{\Gamma_I} = 0, \forall g \in H_0^{\frac{1}{2} + \varepsilon}(\Gamma_I),
\end{equation*}
and so
\begin{equation}
    k_b \partial_{\bm{e}_3} p_b = k_p \partial_s p_p \vert_{\omega_p^+}.
\end{equation}
Note that applying the interface condition \ref{normal}  to equation \ref{l.5} gives, as needed,
$$\rho_b w + D\Delta_{\omega_p}^2 w + \gamma w + \alpha_p \Delta_{\omega_p} \mathcal{K}(p_p) - F_p = \rho_p [f_5 + f_4] \ \text{in} \ L^2(\omega_p)$$
where as always $$F_p\equiv \sigma _{b}\mathbf{e}_{3}\cdot \mathbf{e}_{3}-\sigma _{f} \mathbf{e
}_{3} \cdot \mathbf{e}_{3}.$$
Finally, set
\begin{equation*}
    \bm{\zeta} = \bm{\eta} - \mathbf{f}_1 \in \bm{H}^1_{\#,*}(\Omega_b).
\end{equation*}
This  concludes the maximality argument.

\section{Nonlinear Extension: von K\'arm\'an's Plate}
\label{nonlinear}

In this section we demonstrate well-posedness of mild
solutions to the dynamic problem
\emph{in the presence of the scalar von K\'arm\'an nonlinearity} \cite{springer}
$$f(w) = [w, v(w) + F_0]$$
which is defined in terms of the von K\'arm\'an bracket
$$[u, w] = \partial_{x}^2 u \partial_{y}^2 w + \partial_{y}^2 u \partial_{x}^2 w - 2 \partial_{xy} u \partial_{xy} w$$ and $F_0$ is a given function from $H^4(\Omega)\cap H_0^2(\Omega)$. 
Following \cite{vk}, we restate some basic facts about the von K\'arm\'an nonlinearity taken from \cite{springer}, which leads to the \emph{locally Lipschitz} property of $f$ from $H_{0}^{2}(\omega_p
)\rightarrow L^{2}(\omega_p )$. The latter provides, through standard perturbation results \cite{pazy}, the ability to extend semigroup solutions in the presence of the nonlinear term. 

To obtain the central property, we rely on the modern property known as the  sharp regularity of the Airy stress
function. To wit, this can be found in \cite[Corollary 1.4.4]{springer}. First, denoting the one has the estimate,%
\begin{equation}
\Vert A_C^{-1}[u,w]\Vert _{W^{2,\infty }(\omega_p
)}\leq C\Vert u\Vert _{2,\omega_p }\Vert w\Vert _{2,\omega_p },  \label{bi_D}
\end{equation}%
where $A_C$ denotes the biharmonic operator with clamped
boundary conditions on $\omega_p$. The function $v=v(u)$ is the so called Airy stress function, which satisfies the following BVP
\begin{align*}
    \Delta^2 v + [u,u] = 0;~~~~ 
    v \vert_{\partial \omega_p} = \frac{\partial v}{\partial n} \vert_{\partial \omega_p} = 0.
\end{align*}
Combining the previous two items, we have from (\ref{bi_D}) 
\begin{equation*}
\Vert v(w)\Vert _{W^{2,\infty }(\omega_p )}\leq C\Vert w\Vert _{2,\omega_p }^{2}.
\end{equation*}%
From this, we can show that 
\begin{equation}
\Vert \lbrack u_{1},v(u_{1})]-[u_{2},v(u_{2})]\Vert _{L^{2}(\omega_p )}\leq C%
\Big(\Vert u_{1}\Vert _{H_{0}^{2}(\omega_p )}^{2}+\Vert u_{2}\Vert
_{H_{0}^{2}(\omega_p )}^{2}\Big)\Vert
u_{1}-u_{2}\Vert _{H_{0}^{2}(\omega_p )}  \label{airy-lip}
\end{equation}%
(\cite[Corollary 1.4.5]{springer}). Thus, the nonlinearity $%
f(w)=[w,v(w)+F_{0}]$ is locally Lipschitz from $H_{0}^{2}({\omega_p })$ into $L^{2}(\omega_p )$.

We also mention that the von K\'arm\'an nonlinearity is {\em potential}. Namely, it has the property that there exists a
potential energy functional $\Pi$ associated with $f$. In the case of the
von K\'arm\'an nonlinearity, it has the form 
\begin{equation*}
\Pi(w)=\frac14\int_{\omega_p}\Big(|\Delta v(w)|^2 -2w[w,F_0] \Big) dx,
\end{equation*}
and possesses the properties that $\Pi$ is a $C^1$-functional on $%
H^2_0(\omega_p)$ such that $f$ is a Fr\'echet derivative of $\Pi$: $%
-f(w)=\Pi^{\prime }(w)$. From this it follows that for a smooth function $w$%
: 
\begin{equation*}
\dfrac{d}{dt}\Pi(w)=(\Pi^{\prime }(w),w_t)=-(f(w),w_t)_{\omega_p}.
\end{equation*}
Moreover $\Pi(\cdot)$ is locally bounded on $H^2_0(\omega_p)$, and there exist 
$\eta<1/2$ and $C\ge 0$ such that 
\begin{equation}  \label{8.1.1c1}
\eta \|\Delta w\|_{\omega_p}^2 +\Pi(w)+C \ge 0\;,\quad \forall\, w\in
H^2_0(\omega_p)\;.
\end{equation}

\begin{remark}
We note that the Berger and Kirchhoff nonlinearities, for instance discussed
in \cite{supersonic,Chu2013-comp}, satisfy the above properties; they:
(i) are locally Lipschitz $H_0^2(\omega_p)\to L^2(\omega_p)$, (ii) have a $C^1$
antiderivative $\Pi$ satisfying the above properties.
\end{remark}

We now address the system \eqref{*.a}--\ref{*.l}, where $F_p$ includes the nonlinear restoring force, $f(w)$, as above. That is to say, we take the p
plate nonlinearity on the RHS: 
\begin{equation}
w_{tt}+D \Delta_{\omega_p}^{2}w + \gamma w +  \alpha_p \Delta_{\omega_p} \mathcal{K}(p_p)  =\sigma_b \mathbf{e}_3 \cdot \mathbf{e}_3 \vert_{\omega_p} - \sigma_f \mathbf{e}_3 \cdot \mathbf{e}_3 \vert_{\omega_p}+[w,v(w)+F_{0}]~\text{ on }~\Omega
\times (0,\infty ).
\label{nlp}
\end{equation}%
We will show the well-posedness of \emph{mild solutions} (in the sense of 
\cite{pazy}) in the presence of the von K\'arm\'an nonlinearity. To th\i s end,
we define a nonlinear operator $\mathcal{F}:\mathbf{X}\rightarrow \mathbf{X%
}$, given by 
\begin{equation*}
\mathcal{F}\big(\lbrack \bm{\eta},\bm{\zeta},p_b,w,v,p_p,\mathbf{u}]\big)=[\mathbf{0}, \mathbf{0}, 0, 0, f(w), 0,\mathbf{0}].
\end{equation*}%
This mapping is locally Lipschitz from $\mathbf X \to \mathbf X$ (by the properties of $f$ above), and thus
will be considered as a perturbation to the Cauchy problem (\ref{cauchy}) as shown in (\ref{cauchy_vk})

\begin{theorem}
\label{th:nonlin} The nonlinear Cauchy problem in \eqref{cauchy_vk} is well-posed in the sense of mild solutions. This is to
say: there is a unique local-in-time mild solution $\mathbf{y}(t)$ on $t\in
\lbrack 0,t_{\text{max}})$ (which is also a weak solut\i on). Moreover, for $%
\mathbf{y}_{0}\in {D}(\mathcal{A})$, the corresponding solution is
strong.

In e\i ther case, when $t_{\text{max}}(\mathbf{y}_{0})<\infty $, we have
that $||\mathbf{y}(t)||_{\mathbf{X}}\rightarrow \infty $ as $t\nearrow t_{%
\text{max}}(y_{0})$.
\end{theorem}

\begin{proof}
With the fluid-structure semigroup $e^{\mathcal At}$ in hand from Theorem \ref{th:main1}, this is a direct application of \cite[Theorem 1.4, p.185]{pazy}
and  the localized version of \cite[Theorem 1.6, p.189]{pazy}, pertaining to locally Lipschitz perturbations of semigroup dynamics.  
\end{proof} In order to guarantee global solutions, i.e., valid solutions of (%
\ref{cauchy_vk}) on $[0,T]$ for any $T>0$, we must utilize the
\textquotedblleft good" structure of $\Pi $. The energy identity, in the
presence of nonlinearity (i.e., \eqref{nlp}), is
obtained in a standard way using the properties of $\Pi $ (see \cite%
{supersonic,Chu2013-comp}). Consider $\mathbf{y}_{0}=[\bm{\eta}_0, \bm{\zeta}_0, p_{b,0}, w_0, v_0, p_{p,0}, \mathbf{u}_0]^T \in \mathbf{X}$. Any
mild solution corresponding to (\ref{cauchy_vk}) satisfies: 
\begin{align}
\mathcal{E} &\big(\bm{\eta}(t),\bm{\zeta}(t),p_b(t),w_{t}(t), v(t), p_p(t), \mathbf{u}(t) \big) +\Pi (w(t)) \\
=&~ \mathcal{E}\big(\bm{\eta}_0, \bm{\zeta}_0, p_{b,0}, w_0, v_0, p_{p,0}, \mathbf{u}_0\big)+\Pi (w(0)) \\
&+\frac{1}{2}\int_{0}^{t}\left[ \beta \int_{\Gamma
_{I}}\left\vert \mathbf{u}\cdot \mathbf{\tau }\right\vert ^{2}d\Gamma
_{I}+\left\Vert k_{b}^{\frac{1}{2}}\nabla p_{b}\right\Vert _{\Omega
_{b}}^{2}+\left\Vert k_{p}^{\frac{1}{2}}\partial _{s}p_{p}\right\Vert
_{\Omega _{p}}^{2}+\mu _{f}\left\Vert \nabla \mathbf{u}\right\Vert _{\Omega
_{f}}^{2}\right]d\tau .  \notag
\end{align}%
From this a priori relation, Gronwall's inequality, and the bound on the $%
\Pi$ in \eqref{8.1.1c1}, we observe that the previously obtained {\em local} mild solutions are {\em necessarily} bounded on any interval  $[0,T]$; immediately we obtain the final (nonlinear) well-posedness theorem, stated below. 

\begin{theorem}[Restatement of Theorem \ref{th:main4}.]
For any $T>0$, the Cauchy problem (\ref{cauchy_vk}) is semigroup well-posed on $\mathbf{X}$ for all $[0,T]$. This is to say
that the PDE problem in \eqref{*.a}--\eqref{*.l}, taking into account
the nonlinear plate equation \eqref{nlp}, is well-posed in the sense of
mild solutions for initial data $\mathbf y_0 \in \mathbf X$. 

Moreover, we have the \emph{%
global-in-time} estimate for solutions: 
\begin{equation}
\sup_{t \in [0,\infty)} \mathcal{E}(\bm{\eta}(t), \bm{\zeta}(t),p_b(t),w_{t}(t), v(t), p_p(t), \mathbf{u}(t)) \le 
\mathbf{C}(\bm{\eta}_0, \bm{\zeta}_0, p_{b,0}, w_0, v_0, p_{p,0}, \mathbf{u}_0, F_0).
\end{equation}
\end{theorem}

\section{Future Work}
The main results presented here complement those for semigroup solutions for the Biot-Stokes filtration in \cite{AGW,AW} and for weak solutions in \cite{bcmw}. Namely, we have provided an operator-theoretic framework for the multilayer filtration with a poroelastic plate as the mediating interface. We are then able to demonstrate that the underlying dynamics operator generates a strongly continuous semigroup, which provides strong and semigroup (mild) solutions in the inertial and non-degenerate regimes. With the semigroup in hand, we are able to incorporate and accommodate large deflection nonlinearity   into the solution framework via perturbation theory. 

In this brief section we try to contextualize our result from the point of view of now-possible future work, as well as the usefulness of our work in other multiphysics and/or poroelastic research platforms.

\begin{itemize}
\item First, with a semigroup framework in hand, we can proceed to investigate both regularity and stability properties of the solution. 
The former results from invoking elliptic regularity results---when possible---to obtain clear regularity classes for the coupled dynamics in the case of strong solutions. With this information, nonlinear extensions involving bulk poroelasticity, as well as fluid nonlinearity, can be considered (likely from the fixed point point of view---see \cite{BMW,bgsw,bw}). 

For stability, the resolvent problem formulated here allows us to begin spectral analysis to determine whether the dissipating and coupling in the problem is sufficient to yield asymptotic stability, or perhaps, a stronger notion of stability. While it is unlikely that uniform stability holds in these coupled scenarios---see \cite{georgemulti,AT}---we believe that the dissipation in the problem is inherently strong enough to provide some decay. 

Finally, with the related results for Biot-Stokes filtrations established in \cite{AGW}, both of the aforementioned questions can be studied comparatively across the two filtration dynamics---Biot-Stokes and multilayer. In this way, the presence of the poroelastic plate as the interface can be studied from the point of view of it being (de)-regularizing and/or (de)-stabilizing. 

\item We freely acknowledge that the toroidal domain invoked here is chosen for mathematical convenience. In particular, the technical issues associated with boundary triple points, corners, and interface edges, are substantial. However, a forthcoming work will exploit certain recent observations on 3D coupled systems \cite{georgeelim} which should permit the resolution and definition of interface traces of higher order, which, heretofore, have been technically insurmountable. The central issue revolves around the resolution of the pressure, and the ability to define a certain elliptic problem in the fluid domain which operates at very low regularity. 

We believe that the results herein can be appropriately adapted (for strong and semigroup) solutions in the near term, bearing in mind the modifications in regularity introduced by the non-separation of boundary components, as well as the need to interpret certain higher-order traces in the generalized sense, even for strong solutions. 
\item Although our work here deals with the linear poroplate, as derived in \cite{mikelic} and studied in \cite{ellie,bcmw}, there are groups interested in nonlinear extensions. In particular, for moving domains (as in the FSI literature, \v{C}ani\'c et al.), the interface should not be treated in a static way. Hence, a large deflection plate model is warranted. On the other hand, the derivation of the poroplate (using scalings and homogenization), should intrinsically include the possibility of nonlinear perfusion in the nonlinear deflection regime---as these dynamics are necessarily coupled.

In this paper we have demonstrated the ability to incorporate structural nonlinearity in the plate, without regard to the moving interface or internal diffusion of the poroplate. We note that the literature which invokes a {\em reticular plate} might easily allow for plate nonlinearity as well, but does not contain a perfusive dynamics. Hence, we believe two natural ways to extend the results here are: (1) derive and analyze a large deflection model which permits perfusion on the deformed domain; (2) incorporate such a model in to a moving domain framework (as in \cite{jeff1,jeff2,jeff3}), where the large poroplate deflections indeed dictate the interface, and, as such, the time-evolving fluid and bulk poroelastic domains. 
\end{itemize}

\begin{appendix}

\begin{section}{Babu\v{s}ka-Brezzi Lemma \texorpdfstring{\cite{kesavan}}{}} \label{sec:babuksa-brezzi}
    \begin{theorem}{Babu\v{s}ka-Brezzi}
        Let $\Sigma$ and $V$ be real Hilbert spaces. Let $a : \Sigma \times \Sigma \to \mathbb{R}$ and $b : \Sigma \times V \to \mathbb{R}$ be continuous bilinear forms. Let
        $$Z = \{ \sigma \in \Sigma ~|~ b(\sigma, v) = 0, \forall v \in V \}.$$
        Assume $a$ is Z-elliptic, and $b$ satisfies the Babu\v{s}ka-Brezzi condition:
        There exists $\beta > 0$ such that
        $$\sup_{\tau \in \Sigma \setminus \{ 0 \}} \frac{b(\tau,v)}{\Vert \tau \Vert_{\Sigma}} \geq \beta \Vert v \Vert_{V}, \forall v \in V.$$
        Then given $\kappa \in \Sigma$ and $l \in V$, there exists unique $(\sigma, u) \in \Sigma \times V$ such that
        \begin{align*}
            a(\sigma, \tau) + b(\tau, u) &= (\kappa, \tau)_{\Sigma} &\forall \tau \in \Sigma \\
            b(\sigma, v) &= (l,v)_V &\forall v \in V.
        \end{align*}
    \end{theorem}
\end{section}
    
\end{appendix}

\bibliographystyle{plain} \footnotesize
\bibliography{refs_updated.bib}                                                                                                                                                                                                                                             

\end{document}